\documentclass[a4,12pt]{amsart}

\usepackage{amssymb}
\usepackage{amstext}
\usepackage{amsmath}
\usepackage{amscd}
\usepackage{amsthm}
\usepackage{amsfonts}
\usepackage{enumerate}
\usepackage{graphicx}
\usepackage{latexsym}

\newtheorem{thm}{Theorem}[section]
\newtheorem*{mthm}{Theorem I}
\newtheorem*{mthm2}{Theorem II}

\newtheorem{cor}[thm]{Corollary}
\newtheorem{lem}[thm]{Lemma}
\newtheorem{prop}[thm]{Proposition}
\newtheorem*{thm*}{Theorem}
\newtheorem*{cor*}{Corollary}
\theoremstyle{definition}
\newtheorem{dfn}[thm]{Definition}
\newtheorem{rem}[thm]{Remark}
\newtheorem{ques}[thm]{Question}
\newtheorem{conj}[thm]{Conjecture}

\newtheorem{conv}[thm]{Convention}

\newtheorem*{claim*}{Claim}

\newtheorem{nota}[thm]{Notation}

\numberwithin{equation}{thm}
\def\depth{\operatorname{{\sf depth}}}
\def\CIdim{\operatorname{{\sf CIdim}}}
\def\Gdim{\operatorname{{\sf Gdim}}}
\def\dim{\operatorname{{\sf dim}}}
\def\size{\operatorname{\mathsf{size}}}
\def\rank{\operatorname{\mathsf{rank}}}

\def\radi{\operatorname{{\sf radius}}}
\def\radius{\operatorname{{\sf radius}}}

\def\height{\operatorname{{\sf ht}}}
\def\len{\operatorname{{\sf length}}}
\def\pd{\operatorname{{\sf pd}}}
\def\LL{\operatorname{{\sf \ell\ell}}}

\def\cx{\operatorname{{\sf cx}}}

\def\codim{\operatorname{{\sf codim}}}
\def\mult{\operatorname{{\sf mult}}}
\def\res{\operatorname{res}}
\def\thick{\operatorname{thick}}
\def\add{\operatorname{add}}

\def\mod{\operatorname{mod}}

\def\CM{\operatorname{CM}}
\def\TR{\operatorname{\mathcal{G}}}
\def\uCM{\operatorname{\underline{CM}}}
\def\cm{\operatorname{\mathsf{CM}}}

\def\Spec{\operatorname{Spec}}
\def\Sing{\operatorname{Sing}}
\def\NF{\operatorname{NF}}

\def\Min{\operatorname{\mathsf{Min}}}
\def\spec{\operatorname{\mathsf{Spec}}}
\def\sing{\operatorname{\mathsf{Sing}}}

\def\ZZ{\mathbb{Z}}
\def\N{\mathbb{N}}

\def\C{\mathcal{C}}
\def\A{\mathcal{A}}
\def\X{\mathcal{X}}
\def\Y{\mathcal{Y}}
\def\Z{\mathcal{Z}}

\def\m{\mathfrak{m}}
\def\n{\mathfrak{n}}
\def\p{\mathfrak{p}}
\def\q{\mathfrak{q}}
\def\r{\mathfrak{r}}
\def\a{\mathfrak{a}}
\def\v{\mathsf{V}}

\def\syz{\mathrm{\Omega}}
\def\tr{\mathrm{Tr}}

\def\Ext{\operatorname{Ext}}

\def\Hom{\operatorname{Hom}}

\newcommand{\ses}[3]{0 \to {#1} \to {#2} \to {#3} \to 0}
\def\uac{{\bf (UAC)}}
\def\Ann{\operatorname{Ann}}

\def\im{\operatorname{Im}}
\begin{document}
\setlength{\baselineskip}{15pt}
\title{The radius of a subcategory of modules}
\author{Hailong Dao}
\address{Department of Mathematics, University of Kansas, Lawrence, KS 66045-7523, USA}
\email{hdao@math.ku.edu}
\urladdr{http://www.math.ku.edu/~hdao/}
\author{Ryo Takahashi}
\address{Graduate School of Mathematics, Nagoya University, Furocho, Chikusaku, Nagoya 464-8602, Japan/Department of Mathematics, University of Nebraska, Lincoln, NE 68588-0130, USA}
\email{takahashi@math.nagoya-u.ac.jp}
\urladdr{http://www.math.nagoya-u.ac.jp/~takahashi/}
\dedicatory{Dedicated to Professor Craig Huneke on the occasion of his sixtieth birthday}
\thanks{2010 {\em Mathematics Subject Classification.} Primary 13C60; Secondary 13C14, 16G60, 18E30}
\thanks{{\em Key words and phrases.} resolving subcategory, thick subcategory, Cohen-Macaulay module, complete intersection, dimension of triangulated category, Cohen-Macaulay representation type}
\thanks{The first author was partially supported by NSF grant DMS 0834050. The second author was partially supported by JSPS Grant-in-Aid for Young Scientists (B) 22740008 and by JSPS Postdoctoral Fellowships for Research Abroad}
\begin{abstract}
We introduce a new invariant for subcategories $\X$ of finitely generated modules over a local ring $R$ which we call the radius of $\X$.
We show that if $R$ is a complete intersection and $\X$ is resolving, then finiteness of the radius forces $\X$ to contain only maximal Cohen-Macaulay modules. We also show that the category of maximal Cohen-Macaulay modules has finite radius when $R$ is a Cohen-Macaulay complete local ring with perfect coefficient field. 
We link the radius to many well-studied notions such as the dimension of the stable category of maximal Cohen-Macaulay modules, finite/countable Cohen-Macaulay representation type and the uniform Auslander condition.
\end{abstract}
\maketitle
\section*{Introduction}

Let $R$ be a commutative Noetherian local ring and $\mod R$  the category of finitely generated modules over $R$.
In this paper we introduce and study a new invariant for subcategories $\X$ of $\mod R$ which we call the radius of $\X$.
Roughly speaking, it is defined as the least number of extensions necessary to build the whole objects in $\X$ out of a single object in $\mod R$.
(For the precise definition, see Definition \ref{defrad} in this paper.)
Our definition is inspired by the notion of dimension of triangulated categories that was introduced by Rouquier in \cite{R}.

We obtain strong evidences that the concept of radius is intimately linked to both the representation theory and the singularity of $R$.
For example, over a Gorenstein complete local ring $R$, the category of maximal Cohen-Macaulay modules has radius zero if and only if $R$ has finite Cohen-Macaulay representation type, in other words, $R$ is a simple hypersurface singularity (when $R$ has an algebraically closed coefficient field of characteristic zero).
In addition, the category of maximal Cohen-Macaulay modules over a complete local hypersurface (over an algebraically closed field of characteristic not two) of countable Cohen-Macaulay representation type has radius one.
We also observe a tantalizing connection to the uniform Auslander condition, which has attracted researchers over the years.

Perhaps most surprisingly, one of the corollaries of our first main result (Theorem \ref{3}) states:

\begin{mthm}
Let $R$ be a local complete intersection, and let $\X$ be a resolving subcategory of $\mod R$. If the radius of $\X$ is finite, then $\X$ contains only maximal Cohen-Macaulay modules.
\end{mthm}

We conjecture that the above result holds for all Cohen-Macaulay local rings. Our second main result below supports this conjecture, which follows from a more general results (Theorems \ref{88} and \ref{13}).

\begin{mthm2}
Let $R$ be a Cohen-Macaulay complete local ring with perfect coefficient field.
Then the category of maximal Cohen-Macaulay modules over $R$ has finite radius.
\end{mthm2}

The structure of the paper is as follows.
In Section \ref{prelim} we set the basic notations and definitions.
Section \ref{defisec} contains our key definition (Definition \ref{keydef}) of the radius of a subcategory of $\mod R$, as well as some detailed comparisons to similar notions.
We also give several results connecting the radius to the singularities of finite and countable Cohen-Macaulay representation type.
Sections \ref{mainsec} and \ref{proofsec} consist of the statement and proof of our main Theorem I, respectively.
We also discuss here thickness of resolving subcategories of maximal Cohen-Macaulay modules over a complete intersection. Section \ref{section5} contains the proof of (generalizations of) our main Theorem II.
Section \ref{appsec} connects  the main results  to  the uniform Auslander condition and discuss some open questions.

\section{Preliminaries}\label{prelim}

In this section, we recall the definitions of a resolving subcategory, totally reflexive modules and a thick subcategory.
We begin with our convention.

\begin{conv}
Throughout this paper, we assume all rings are commutative Noetherian rings with identity.
All modules are finitely generated.
All subcategories are full and strict.
(Recall that a subcategory $\X$ of a category $\C$ is called {\it strict} provided that for objects $M,N\in\C$ with $M\cong N$, if $M$ is in $\X$, then so is $N$.)
Hence, the {\em subcategory} of a category $\C$ consisting of objects $\{M_\lambda\}_{\lambda\in\Lambda}$ always mean the smallest strict full subcategory of $\C$ to which $M_\lambda$ belongs for all $\lambda\in\Lambda$.
Note that this coincides with the full subcategory of $\C$ consisting of all objects $X\in\C$ such that $X\cong M_\lambda$ for some $\lambda\in\Lambda$.
Let $R$ be a (commutative Noetherian) ring.
Denote by $\mod R$ the category of (finitely generated) $R$-modules and $R$-homomorphisms.
For a Cohen-Macaulay local ring $R$, we call a maximal Cohen-Macaulay $R$-module just a Cohen-Macaulay $R$-module.
We denote by $\CM(R)$ the subcategory of $\mod R$ consisting of Cohen-Macaulay $R$-modules.
\end{conv}

The following notation is used throughout this paper.

\begin{nota}
For a subcategory $\X$ of $\mod R$, we denote by $\add\X$ (or $\add_R\X$) the {\em additive closure} of $\X$, namely, the subcategory of $\mod R$ consisting of direct summands of finite direct sums of modules in $\X$.
When $\X$ consists of a single module $M$, we simply denote it by $\add M$ (or $\add_RM$).
For an $R$-module $M$, we denote by $M^*$ the $R$-dual module $\Hom_R(M,R)$.
For a homomorphism $f:M\to N$ of $R$-modules, $f^*$ denotes the $R$-dual homomophism $N^*\to M^*$ sending $\sigma\in N^*$ to the composition $\sigma\cdot f\in M^*$.
\end{nota}

The notion of a resolving subcategory has been introduced by Auslander and Bridger \cite{AB}.
It can actually be defined for an arbitrary abelian category with enough projective object.
The only resolving subcategories we deal with in this paper are ones of $\mod R$.

\begin{dfn}
A subcategory $\X$ of $\mod R$ is called {\it resolving} if the following hold.
\begin{enumerate}[(R1)]
\item
$\X$ contains the projective $R$-modules.
\item
$\X$ is closed under direct summands: if $M$ is an $R$-module in $\X$ and $N$ is an $R$-module that is a direct summand of $M$, then $N$ is also in $\X$.
\item
$\X$ is closed under extensions: for an exact sequence $0 \to L \to M \to N \to 0$ of $R$-modules, if $L,N$ are in $\X$, then so is $M$.
\item
$\X$ is closed under kernels of epimorphisms: for an exact sequence $0 \to L \to M \to N \to 0$ of $R$-modules, if $M,N$ are in $\X$, then so is $L$.
\end{enumerate}
\end{dfn}

A resolving subcategory is a subcategory such that any two minimal resolutions of a module by modules in it have the same length; see \cite[Lemma (3.12)]{AB}.
Note that one can replace the condition (R1) with:
\begin{enumerate}[(R1')]
\item
$\X$ contains $R$.
\end{enumerate}

Next we recall the notion of a totally reflexive module.

\begin{dfn}
An $R$-module $M$ is called {\em totally reflexive} if the natural homomorphism $M\to M^{**}$ is an isomorphism and $\Ext_R^i(M,R)=0=\Ext_R^i(M^*,R)$ for all $i>0$.
We denote by $\TR(R)$ the subcategory of $\mod R$ consisting of totally reflexive modules.
\end{dfn}

A totally reflexive module was defined by Auslander \cite{A}, and deeply studied by Auslander and Bridger \cite{AB}.
The $R$-dual of a totally reflexive $R$-module is also totally reflexive.
Every projective module is totally reflexive, i.e., $\add R\subseteq\TR(R)$.
If $R$ is a Cohen-Macaulay local ring, then every totally reflexive $R$-module is Cohen-Macaulay, i.e., $\TR(R)\subseteq\CM(R)$.
When $R$ is a Gorenstein local ring, an $R$-module is totally reflexive if and only if it is Cohen-Macaulay, i.e., $\TR(R)=\CM(R)$.
For more details, see \cite{AB} and \cite{Ch}.

Syzygies, cosyzygies and transposes are key tools in this paper.
We recall here their precise definitions.

\begin{dfn}\label{stc}
Let $(R,\m)$ be a local ring, and let $M$ be an $R$-module.
\begin{enumerate}[(1)]
\item
Take a minimal free resolution
$
\cdots \xrightarrow{\delta_{n+1}} F_n \xrightarrow{\delta_n} F_{n-1} \xrightarrow{\delta_{n-1}} \cdots \xrightarrow{\delta_1} F_0 \to M \to 0
$
of $M$.
Then, for each $n\ge1$, the image of $\delta_n$ is called the {\em $n$-th syzygy} of $M$ and denoted by $\syz^nM$ (or $\syz_R^nM$).
For convention, we set $\syz^0M=M$.
\item
The cokernel of the $R$-dual map $\delta_1^*:F_0^*\to F_1^*$ is called the {\em (Auslander) transpose} of $M$ and denoted by $\tr M$ (or $\tr_RM$).
\item
Let
$
0 \to M \to F_{-1} \xrightarrow{\delta_{-1}} \cdots \xrightarrow{\delta_{-(n-1)}} F_{-n} \xrightarrow{\delta_{-n}} F_{-(n+1)} \xrightarrow{\delta_{-(n+1)}} \cdots
$
be a {\em minimal free coresolution} of $M$, that is, an exact sequence with $F_{-n}$ free and $\im\delta_{-n}\subseteq\m F_{-(n+1)}$ for all $n\ge1$.
Then we call the image of $\delta_{-n}$ the {\em $n$-th cosyzygy} of $M$ and denote it by $\syz^{-n}M$ (or $\syz_R^{-n}M$).
\end{enumerate}
\end{dfn}

Let $R$ be a local ring.
Then by \cite[Lemma 3.2]{Y2} one can replace (R4) with:
\begin{enumerate}[(R4')]
\item
$\X$ is closed under syzygies: if $M$ is in $\X$, then so is $\syz M$.
\end{enumerate}

Totally reflexive modules behave well under taking their syzygies, cosyzygies and transposes.
Let $R$ be a local ring.
Let $M$ be a totally reflexive $R$-module.
The $R$-dual of a minimal free resolution (respectively, coresolution) of $M$ is a minimal free coresolution (respectively, resolution) of $M^*$.
In particular, a minimal free coresolution of $M$ always exists, and it is uniquely determined up to isomorphism.
The $n$-th syzygy $\syz^nM$ and cosyzygy $\syz^{-n}M$ are again totally reflexive for all $n$.
This is an easy consequence of \cite[(1.2.9) and (1.4.8)]{Ch}.
The transpose $\tr M$ is also totally reflexive; see \cite[Proposition (3.8)]{AB}.
For an $R$-module $M$, the $n$-th syzygy $\syz^nM$ for any $n\ge1$ and the transpose $\tr M$ are uniquely determined up to isomorphism, since so are a minimal free resolution of $M$.
If $M$ is totally reflexive, then the $n$-th cosyzygy $\syz^{-n}M$ for any $n\ge1$ is also uniquely determined up to isomorphism, since so is a minimal free coresolution of $M$.

A lot of subcategories of $\mod R$ are known to be resolving.
For example, $\CM(R)$ is a resolving subcategory of $\mod R$ if $R$ is Cohen-Macaulay.
The subcategory of $\mod R$ consisting of totally reflexive $R$-modules is resolving by \cite[(3.11)]{AB}.
One can construct a resolving subcategory easily by using vanishing of Tor or Ext.
Also, the modules of complexity less than a fixed integer form a resolving subcategory of $\mod R$.
For the details, we refer to \cite[Example 2.4]{res}.

Now we define a thick subcategory of totally reflexive modules.

\begin{dfn}
A subcategory $\X$ of $\TR(R)$ is called {\it thick} if it is closed under direct summands, and short exact sequences: for an exact sequence $0 \to L \to M \to N \to 0$ of totally reflexive $R$-modules, if two of $L,M,N$ are in $\X$, then so is the third.
\end{dfn}

A typical example of a thick subcategory is obtained by restricting a resolving subcategory to $\TR(R)$.

The following proposition is shown by an argument dual to \cite[Lemma 3.2]{Y2}.

\begin{prop}\label{12}
Let $R$ be a local ring.
Let $\X$ be a subcategory of $\TR(R)$ containing $R$.
Then $\X$ is a thick subcategory of $\TR(R)$ if and only if $\X$ is a resolving subcategory of $\mod R$ and is closed under cosyzygies: if $M$ is in $\X$, then so is $\syz^{-1}M$.
\end{prop}

Let $(R,\m)$ be a local ring.
We call $R$ a {\em hypersurface} if the $\m$-adic completion $\widehat R$ of $R$ is a residue ring of a complete regular local ring by a principal ideal.
We say that $R$ is a {\em complete intersection} if $\widehat R$ is a residue ring of a complete regular local ring by an ideal generated by a regular sequence.

We recall the definitions of {\em Gorenstein dimension} and {\em complete intersection dimension}, which are abbreviated to G-dimension and CI-dimension.
These notions have been introduced by Auslander and Bridger \cite{AB} and Avramov, Gasharov and Peeva \cite{AGP}, respectively.

\begin{dfn}
Let $R$ be a local ring, and let $M$ be an $R$-module.
The {\em G-dimension} of $M$, denoted $\Gdim_RM$, is defined as the infimum of the lengths of totally reflexive resolutions of $M$, namely, exact sequences of the form
$
0 \to X_n \to X_{n-1} \to \cdots \to X_1 \to X_0 \to M \to 0
$
with each $X_i$ being totally reflexive.
The {\em CI-dimension} of $M$ is defined as the infimum of
$
\pd_S (M\otimes_RR')-\pd_SR'
$
where $R \to R' \leftarrow S$ runs over the quasi-deformations of $R$.
Here, a diagram $R \overset{f}{\to} R' \overset{g}{\leftarrow} S$ of homomorphisms of local rings is called a quasi-deformation of $R$ if $f$ is faithfully flat and $g$ is a surjection whose kernel is generated by an $S$-sequence.
\end{dfn}

Recall that $M$ is said to have {\em complexity} $c$, denoted by $\cx_RM=c$, if $c$ is the least nonnegative integer $n$ such that there exists a real number $r$ satisfying the inequality $\beta_i^R(M)\le ri^{n-1}$ for all $i\gg 0$.

\begin{rem}
For a local ring $(R,\m,k)$ and a module $M$ over $R$, the following are known to hold.
For the proofs, we refer to \cite{Ch} and \cite{AGP}.
\begin{enumerate}[(1)]
\item
$\Gdim_RM=\infty$ if and only if $M$ does not admit a totally reflexive resolution of finite length.
\item
$\CIdim_RM=\infty$ if and only if $\pd_S(M\otimes_RR')=\infty$ for every quasi-deformation $R \to R' \leftarrow S$.
\item
One has $M=0\Leftrightarrow\Gdim_RM=-\infty\Leftrightarrow\CIdim_RM=-\infty$.
\item
$\Gdim_RM\le0$ if and only if $M$ is totally reflexive.
\item
If $\Gdim_RM$ (respectively, $\CIdim_RM$) is finite, then it is equal to $\depth R-\depth_RM$.
\item
The inequalities $\Gdim_RM\le\CIdim_RM\le\pd_RM$ hold, and equalities hold to the left of any finite dimension.
\item
If $M\ne0$, then $\Gdim_R(\syz^nM)=\sup\{\Gdim_RM-n,0\}$ and $\CIdim_R(\syz^nM)=\sup\{\CIdim_RM-n,0\}$ hold for all $n\ge0$.
\item
If $R$ is a Gorenstein ring (respectively, a complete intersection), then $\Gdim_RM$ (respectively, $\CIdim_RM$) is finite.
If $\Gdim_Rk$ (respectively, $\CIdim_Rk$) is finite, then $R$ is a Gorenstein ring (respectively, a complete intersection).
\item
If $\CIdim_RM<\infty$, then $\cx_RM<\infty$.
\end{enumerate}
\end{rem}

\section{Definition of the radius of a subcategory}\label{defisec}

This section contains the key definition and establishes several results.
More precisely, we will give the definition of the radius of a subcategory of $\mod R$ for a local ring $R$, and compare it with other notions, such as the dimension of a triangulated category defined by Rouquier.
We will also explore its relationships with representation types of a Cohen-Macaulay local ring.

\begin{dfn}\label{keydef}
Let $R$ be a local ring.
\begin{enumerate}[(1)]
\item
For a subcategory $\X$ of $\mod R$ we denote by $[\X]$ the additive closure of the subcategory of $\mod R$ consisting of $R$ and all modules of the form $\syz^iX$, where $i\ge0$ and $X\in\X$.
When $\X$ consists of a single module $X$, we simply denote it by $[X]$.
\item
For subcategories $\X,\Y$ of $\mod R$ we denote by $\X\circ\Y$ the subcategory of $\mod R$ consisting of the $R$-modules $M$ which fits into an exact sequence
$
0 \to X \to M \to Y \to 0
$
with $X\in\X$ and $Y\in\Y$.
We set $\X\bullet\Y=[[\X]\circ[\Y]]$.
\item
Let $\C$ be a subcategory of $\mod R$.
We define the {\it ball of radius $r$ centered at $\C$} as
$$
[\C]_r=
\begin{cases}
[\C] & (r=1),\\
[\C]_{r-1}\bullet\C=[[\C]_{r-1}\circ[\C]] & (r\ge2).
\end{cases}
$$
If $\C$ consists of a single module $C$, then we simply denote $[\C]_r$ by $[C]_r$, and call it the ball of radius $r$ centered at $C$.
We write $[\C]_r^R$ when we should specify that $\mod R$ is the ground category where the ball is defined.
\end{enumerate}
\end{dfn}

Some similar notions have already been introduced.
In \cite[Definition 3.1]{res} the second author defines the subcategory $\res^n\X$ of the resolving closure $\res\X$ of a given subcategory $\X$ of $\mod R$.
This is different from ours in that $\res^n\X$ is not closed under syzygies.
In \cite{ABIM} the {\em thickening} $\thick^n\X$ of a given subcategory $\X$ of a triangulated category is defined.
This cannot be applied directly to a module category.

\begin{prop}\label{1}
Let $R$ be a local ring.
\begin{enumerate}[\rm(1)]
\item
Let $\X,\Y$ be subcategories of $\mod R$.
The following are equivalent for an $R$-module $M$:
\begin{enumerate}[\rm(a)]
\item
$M$ belongs to $\X\bullet\Y$;
\item
There exists an exact sequence $0 \to X \to Z \to Y \to 0$ of $R$-modules with $X\in[\X]$ and $Y\in[\Y]$ such that $M$ is a direct summand of $Z$.
\end{enumerate}
\item
For subcategories $\X,\Y,\Z$ of $\mod R$, one has $(\X\bullet\Y)\bullet\Z=\X\bullet(\Y\bullet\Z)$.
\item
Let $\C$ be a subcategory of $\mod R$, and let $a,b$ be positive integers.
Then one has $[\C]_a\bullet[\C]_b=[\C]_{a+b}=[\C]_b\bullet[\C]_a$.
\end{enumerate}
\end{prop}

\begin{proof}
(1) The implication (b) $\Rightarrow$ (a) is obvious.
To prove the opposite implication (a) $\Rightarrow$ (b), let $M$ be an $R$-module in $\X\bullet\Y=[[\X]\circ[\Y]]$.
By definition, $M$ is isomorphic to a direct summand of $R^{\oplus p}\oplus\bigoplus_{i=0}^n(\syz^iZ_i)^{\oplus q_i}$, where $p,q_i\ge0$ and $Z_i\in[\X]\circ[\Y]$.
For each $0\le i\le n$ there is an exact sequence $0 \to X_i \to Z_i \to Y_i \to 0$ with $X_i\in[\X]$ and $Y_i\in[\Y]$.
Taking syzygies and direct sums, we have an exact sequence
$$
0 \to R^{\oplus p}\oplus\bigoplus_{i=0}^n(\syz^iX_i)^{\oplus q_i} \to R^{\oplus p}\oplus\bigoplus_{i=0}^n(\syz^iZ_i)^{\oplus q_i}\oplus R^{\oplus r} \to \bigoplus_{i=0}^n(\syz^iY_i)^{\oplus q_i} \to 0.
$$
The left and right terms are in $[\X]$ and $[\Y]$, respectively.
The middle term contains an $R$-module isomorphic to $M$.
Thus the statement (b) follows.

(2) First, let $M$ be an $R$-module in $(\X\bullet\Y)\bullet\Z$.
By the assertion (1) there is an exact sequence $0 \to W \xrightarrow{f} V \to Z \to 0$ with $W\in\X\bullet\Y$ and $Z\in[\Z]$ such that $M$ is a direct summand of $V$.
By (1) again, we have an exact sequence $0 \to X \to U \to Y \to 0$ with $X\in[\X]$ and $Y\in[\Y]$ such that $W$ is a direct summand of $U$.
Writing $U=W\oplus W'$, we make the following pushout diagram.
$$
\begin{CD}
@. 0 @. 0 \\
@. @VVV @VVV \\
@. X @= X \\
@. @VVV @VVV \\
0 @>>> W\oplus W' @>{
\left(
\begin{smallmatrix}
f & 0 \\
0 & 1
\end{smallmatrix}
\right)
}>> V\oplus W' @>>> Z @>>> 0 \\
@. @VVV @VVV @| \\
0 @>>> Y @>>> T @>>> Z @>>> 0 \\
@. @VVV @VVV \\
@. 0 @. 0
\end{CD}
$$
The bottom row implies that $T$ is in $\Y\bullet\Z$, and it follows from the middle column that $M$ belongs to $\X\bullet(\Y\bullet\Z)$.
Hence we have $(\X\bullet\Y)\bullet\Z\subseteq\X\bullet(\Y\bullet\Z)$.

Next, let $M$ be an $R$-module in $\X\bullet(\Y\bullet\Z)$.
Then it follows from (1) that there is an exact sequence $0 \to X \to V \xrightarrow{f} W \to 0$ with $X\in[\X]$ and $W\in\Y\bullet\Z$ such that $M$ is a direct summand of $V$.
Applying (1) again, we have an exact sequence $0 \to Y \to U \to Z \to 0$ with $Y\in[\Y]$ and $Z\in[\Z]$ such that $W$ is a direct summand of $U$.
Write $U=W\oplus W'$, and we have a pullback diagram:
$$
\begin{CD}
@. @. 0 @. 0 \\
@. @. @VVV @VVV \\
0 @>>> X @>>> T @>>> Y @>>> 0 \\
@. @| @VVV @VVV \\
0 @>>> X @>>> V\oplus W' @>{
\left(
\begin{smallmatrix}
f & 0 \\
0 & 1
\end{smallmatrix}
\right)
}>> W\oplus W' @>>> 0 \\
@. @. @VVV @VVV \\
@. @. Z @= Z \\
@. @. @VVV @VVV \\
@. @. 0 @. 0
\end{CD}
$$
We see from the first row that $T$ is in $\X\bullet\Y$, and from the middle column that $M$ is in $(\X\bullet\Y)\bullet\Z$.
Therefore $\X\bullet(\Y\bullet\Z)\subseteq(\X\bullet\Y)\bullet\Z$ holds.

(3) It is enough to show the equality $[\C]_a\bullet[\C]_b=[\C]_{a+b}$.
We prove this by induction on $b$.
It holds by definition when $b=1$.
Let $b\ge2$.
Then we have
$$
[\C]_a\bullet[\C]_b
= [\C]_a\bullet([\C]_{b-1}\bullet\C)
= ([\C]_a\bullet[\C]_{b-1})\bullet\C
= [\C]_{a+b-1}\bullet\C
= [\C]_{a+b},
$$
where the second equality follows from (2), and the induction hypothesis implies the third equality.
\end{proof}

Let $\C$ be a subcategory of $\mod R$ and $r\ge0$ an integer.
By the second and third assertions of Proposition \ref{1}, without danger of confusion we can write:
$$
[\C]_r=\overbrace{\C\bullet\cdots\bullet\C}^r.
$$

Now we can make the definition of the radius of a subcategory.

\begin{dfn}\label{defrad}
Let $R$ be a local ring, and let $\X$ be a subcategory of $\mod R$.
We define the {\it radius} of $\X$, denoted by $\radi\X$, as the infimum of the integers $n\ge0$ such that there exists a ball of radius $n+1$ centered at a module containing $\X$.
By definition, $\radi\X\in\N\cup\{\infty\}$.
\end{dfn}

The definition of the radius of a resolving subcategory looks similar to that of the {\em dimension} of a triangulated category which has been introduced by Rouquier (cf. \cite[Definition 3.2]{R}).
The stable category $\uCM(R)$ of Cohen-Macaulay modules over a Gorenstein local ring $R$ is triangulated by \cite{B,H}, and the dimension of $\uCM(R)$ in the sense of Rouquier is defined.
It might look the same as the radius of $\CM(R)$ in our sense.

However there are (at least) two differences in the definitions:
\begin{enumerate}[(1)]
\item
A defining object for $\dim\uCM(R)$ is required to be inside the category $\uCM(R)$, but a defining object for $\radi\CM(R)$ is not, i.e., it is enough to be an object of $\mod R$.
More precisely, $\dim\uCM(R)$ (respectively, $\radi\CM(R)$) is defined as the infimum of the integers $n\ge0$ such that $\uCM(R)=\langle G\rangle_{n+1}$ for some object $G$ (respectively, $\CM(R)\subseteq[C]_{n+1}$ for some object $C$).
Then $G$ must be an object of $\uCM(R)$, while $C$ may not be an object of $\CM(R)$, just being an object of $\mod R$.
\item
Let $\X$ and $\Y$ be subcategories of $\CM(R)$ and $\uCM(R)$, respectively.
Then the subcategory $\langle\Y\rangle$ of $\uCM(R)$ is closed under taking cosyzygies of Cohen-Macaulay modules in it, but the subcategory $[\X]$ of $\CM(R)$ is not in general.
(In fact, this difference is a reason why we can prove Proposition \ref{10} below but do not know whether the analogue for dimension holds or not; see Question \ref{11} below.)
\end{enumerate}

Thus these two notions are different, but they are still related to each other.
In fact, we can show that the following relationship exists between them.

\begin{prop}\label{dimrad}
Let $R$ be a Gorenstein local ring.
\begin{enumerate}[\rm(1)]
\item
One has the inequality
$
\dim\uCM(R)\le\radi\CM(R).
$
\item
The equality holds if $R$ is a hypersurface.
\end{enumerate}
\end{prop}

\begin{proof}
(1) We may assume that $n:=\radi\CM(R)<\infty$.
Then there exists an $R$-module $C$ such that $\CM(R)$ is contained in the ball $[C]_{n+1}$.

We claim that $\CM(R)=[\syz^{-d}\syz^dC]_{n+1}$ holds, where $d=\dim R$.
This claim implies $\uCM(R)=\langle\syz^dC\rangle_{n+1}$, which shows $\dim\uCM(R)\le n$.

In the following, we show this claim.
Since $\syz^{-d}\syz^dC$ is a Cohen-Macaulay $R$-module, and $\CM(R)$ is a resolving subcategory of $\mod R$, the inclusion $\CM(R)\supseteq[\syz^{-d}\syz^dC]_{n+1}$ hold.
To get the opposite inclusion, it is enough to prove that for every $m\ge1$ and $M\in[C]_m$ we have $\syz^{-d}\syz^dM\in[\syz^{-d}\syz^d C]_m$.
Let us prove this by induction on $m$.
The case $m=1$ is obvious, so let $m\ge2$.
According to Proposition \ref{1}(1), there is an exact sequence
$
0 \to X \to Y \to Z \to 0
$
of $R$-modules with $X\in[C]_{m-1}$ and $Z\in[C]$ such that $M$ is a direct summand of $Y$.
Taking the $d$-th syzygies, we have an exact sequence $0 \to \syz^dX \to \syz^dY\oplus R^{\oplus l} \to \syz^dZ \to 0$ of Cohen-Macaulay $R$-modules.
Since $R$ is Gorenstein, taking the $d$-th cosyzygies makes an exact sequence
$$
0 \to \syz^{-d}\syz^dX \to \syz^{-d}\syz^dY\oplus R^{\oplus k} \to \syz^{-d}\syz^dZ \to 0
$$
of Cohen-Macaulay modules.
The induction hypothesis implies $\syz^{-d}\syz^dZ\in[\syz^{-d}\syz^dC]$ and $\syz^{-d}\syz^dX\in[\syz^{-d}\syz^dC]_{m-1}$.
Since $\syz^{-d}\syz^dM$ is a direct summand of $\syz^{-d}\syz^dY$, it belongs to $[\syz^{-d}\syz^dC]_m$.

(2) Let $n:=\dim\uCM(R)<\infty$.
We find a Cohen-Macaulay $R$-module $G$ such that $\uCM(R)=\langle G\rangle_{n+1}$.
We want to prove that $\CM(R)=[G]_{n+1}$ holds, and it suffices to show that for every $m\ge1$ and $M\in\langle G\rangle_m$ we have $M\in[G]_m$.
Let us use induction on $m$.
The case $m=1$ follows from the fact that $\syz N\cong\syz^{-1}N$ up to free summand for each Cohen-Macaulay $R$-module $N$, since $R$ is a hypersurface.
When $m\ge2$, there exists an exact triangle
$
X \to Y \to Z \to \Sigma X
$
in $\uCM(R)$ with $X\in\langle G\rangle_{m-1}$ and $Z\in\langle G\rangle$ such that $M$ is a direct summand of $Y$.
Then we have an exact sequence
$
0 \to X \to Y\oplus R^{\oplus h} \to Z \to 0
$
of $R$-modules, and we are done by applying the induction hypothesis.
\end{proof}

In the rest of this section, we will study the relationships between the representation types of a Cohen-Macaulay local ring and the radius of the category of Cohen-Macaulay modules.
Recall that a Cohen-Macaulay local ring $R$ is said to be of {\it finite} (respectively, {\it countable}) {\it Cohen-Macaulay representation type} if $\CM(R)$ has only finitely (respectively, countably but not finitely) many indecomposable modules up to isomorphism.

We can describe the property of finite Cohen-Macaulay representation type in terms of a radius.

\begin{prop}\label{10}
Let $R$ be a Gorenstein Henselian local ring.
The following are equivalent:
\begin{enumerate}[\rm(1)]
\item
One has $\radi\CM(R)=0$;
\item
The ring $R$ has finite Cohen-Macaulay representation type.
\end{enumerate}
\end{prop}

\begin{proof}
(2) $\Rightarrow$ (1):
If $M_1,\dots,M_r$ are the nonisomorphic indecomposable Cohen-Macaulay $R$-modules, then we have $\CM(R)=[M_1\oplus\cdots\oplus M_r]$.

(1) $\Rightarrow$ (2):
There is an $R$-module $C$ satisfying $\CM(R)\subseteq[C]$.
Setting $d=\dim R$, we have $\CM(R)=[\syz^{-d}\syz^dC]$.
Replacing $C$ with $\syz^{-d}\syz^dC$, we may assume that $\CM(R)=[C]$ with $C$ being Cohen-Macaulay.

Note that since $R$ is Henselian, the Krull-Schmidt theorem holds, i.e., each $R$-module uniquely decomposes into indecomposable $R$-modules up to isomorphism.
Let $C_1,\dots,C_n$ be the indecomposable direct summands of $C$.
We may assume that $C=C_1\oplus\cdots\oplus C_n$.
Since $R$ is Gorenstein, taking syzygies preserves indecomposability of nonfree Cohen-Macaulay $R$-modules.
We see that the set of nonisomorphic indecomposable Cohen-Macaulay $R$-modules is
$$
\{R\}\cup\{\,\syz^iC_j\mid i\ge0,\,1\le j\le n\,\}.
$$
We may assume that for all $i\ge0$ and $1\le j\ne j'\le n$ we have $\syz^iC_j\not\cong C_{j'}$, because if $\syz^iC_j\cong C_{j'}$ for some such $i,j,j'$, then we can exclude $C_{j'}$ from $C$.
Now fix an integer $j$ with $1\le j\le n$.
As taking cosyzygies preserves indecomposability of nonfree Cohen-Macaulay $R$-modules, $\syz^{-1}C_j$ is isomorphic to $\syz^aC_b$ for some $a\ge0$ and $1\le b\le n$.
Taking the $a$-th cosyzygies, we have $\syz^{-1-a}C_j\cong C_b$, hence $C_j\cong\syz^{1+a}C_b$.
This forces us to have $b=j$, which says that $C_j$ is periodic.
Hence there are only finitely many indecomposable Cohen-Macaulay $R$-modules.
\end{proof}

\begin{ques}\label{11}
Does the equality in Proposition \ref{dimrad}(1) hold true?
If it does, then Proposition \ref{10} will say that \emph{a Gorenstein Henselian local ring $R$ has finite Cohen-Macaulay representation type if and only if $\dim\uCM(R)\le0$}.
This statement is a partial generalization of Minamoto's theorem \cite[Theorem 0.2]{M}, which asserts that the same statement holds for a finite-dimensional selfinjective algebra over a perfect field, extending Yoshiwaki's recent theorem \cite[Corollary 3.10]{Y}.
\end{ques}

The next result hints at further relationship between finite radius of $\CM(R)$ and more well-known classification of singularities.

\begin{prop}
Let $R$ be a complete local hypersurface over an algebraically closed field of characteristic not two.
Assume that $R$ is of countable Cohen-Macaulay representation type.
Then $\radi\CM(R) =1$.
\end{prop}

\begin{proof}
It follows from \cite[Theorem 1.1]{AIR} that there exists an $R$-module $X$ such that for every indecomposable module $M\in\CM(R)$ there is an exact sequence
$
0 \to L \to M\oplus R^n \to N \to 0
$
with $L,N\in\{0,X,\syz X\}$.
This shows that $\CM(R) = [X]_2$.
Now we see from Proposition \ref{10} that the radius of $\CM(R)$ is equal to one.
\end{proof}

\section{Finiteness of the radius of a resolving subcategory}\label{mainsec} 

In this section we state our guiding conjecture and first main result.

\begin{conj}\label{2}
Let $R$ be a Cohen-Macaulay local ring.
Let $\X$ be a resolving subcategory of $\mod R$ with finite radius.
Then every $R$-module in $\X$ is Cohen-Macaulay.
\end{conj}

\begin{rem}
The converse of Conjecture \ref{2} also seems to be true.
We consider this in Section \ref{section5}.
\end{rem}

Let $\X$ be a subcategory of $\mod R$.
We denote by $\res\X$ (or $\res_R\X$) the {\em resolving closure} of $\X$, namely, the smallest resolving subcategory of $\mod R$ containing $\X$.
If $\X$ consists of a single module $M$, then we simply denote it by $\res M$ (or $\res_RM$).
For a prime ideal $\p$ of $R$, we denote by $\X_\p$ the subcategory of $\mod R_\p$ consisting of all modules of the form $X_\p$, where $X\in\X$.
The first main result of this paper is the following theorem.

\begin{thm}\label{3}
Let $R$ be a commutative Noetherian ring.
Let $\X$ be a resolving subcategory of $\mod R$.
Suppose that there exist a prime ideal $\p$ of $R$ with $\height\p>0$ and an $R_\p$-module $M$ with $0\ne M\in\add_{R_\p}\X_\p$ which satisfy one of the following conditions.
\begin{enumerate}[\rm(1)]
\item\label{i}
$\p M=0$.
\item\label{n}
$0<\Gdim_{R_\p}M=n<\infty$ and $\syz_{R_\p}^{-2}\syz_{R_\p}^nM\in\add_{R_\p}\X_\p$.
\item\label{s}
$0<\Gdim_{R_\p}M<\infty$ and $\res_{R_\p}(\syz_{R_\p}^nM)$ is a thick subcategory of $\TR(R_\p)$ for some $n\ge0$.
\item\label{y}
$0<\CIdim_{R_\p}M<\infty$.
\end{enumerate}
Then $\X$ has infinite radius.
\end{thm}

The proof of this theorem will be given in the next section.
As a direct consequence of the above theorem, we obtain two cases in which our conjecture holds true.

\begin{cor}\label{8}
Conjecture \ref{2} is true if
\begin{enumerate}[\rm(1)]
\item
$R$ is a complete intersection, or
\item
$R$ is Gorenstein, and every resolving subcategory of $\mod R$ contained in $\CM(R)$ is a thick subcategory of $\CM(R)$.
\end{enumerate}
\end{cor}

\begin{proof}
Conjecture \ref{2} trivially holds in the case where $R$ is Artinian, so let $(R,\m,k)$ be a Cohen-Macaulay local ring of positive dimension.
Then we have $\height\m>0$.
Let $\X$ be a resolving subcategory of $\mod R$, and suppose that $\X$ contains a non-Cohen-Macaulay $R$-module $M$.

(1) We have $0<\dim R-\depth_RM=\depth R-\depth_RM=\CIdim_RM<\infty$, and Theorem \ref{3}(4) implies that $\X$ has infinite radius.

(2) We have $0<n:=\dim R-\depth_RM=\depth R-\depth_RM=\Gdim_RM<\infty$.
The module $\syz_R^nM$ is Cohen-Macaulay, and by assumption $\res_R(\syz_R^nM)$ is a thick subcategory of $\CM(R)=\TR(R)$.
Theorem \ref{3}(3) implies that $\X$ has infinite radius.
\end{proof}

\section{Proof of Theorem I}\label{proofsec}

This section is devoted to give the proof of Theorem \ref{3} (hence of Theorem I from Introduction), which we break up into several parts.
Most of them also reveal properties of subcategories of $\mod R$ which are interesting in their own right.

First of all, we make a remark to reduce our theorem to the local case.

\begin{rem}\label{rrr}
Let $C$ be an $R$-module, $n\ge0$ an integer and $\p$ a prime ideal of $R$.
Then for a subcategory $\X$ of $\mod R$ the implication
$$
\X\subseteq[C]_n^R\quad\Rightarrow\quad\add_{R_\p}\X_\p\subseteq[C_\p]_n^{R_\p}
$$
holds.
It follows from \cite[Lemma 4.8]{stcm} that $\add_{R_\p}\X_\p$ is a resolving subcategory of $\mod R_\p$.
(The ring $R$ in \cite[Lemma 4.8]{stcm} is assumed to be local, but its proof does not use this assumption, so it holds for an arbitrary commutative Noetherian ring.)
Hence, to prove Theorem \ref{3}, without loss of generality we can assume $(R,\p)$ is a local ring with $\dim R>0$ and $M$ is an $R$-module with $0\ne M\in\X$.
\end{rem}

\subsection{Proof of Theorem \ref{3}(1)}

First, we investigate the annihilators of torsion submodules.
For an ideal $I$ of $R$ and an $R$-module $M$, we denote by $\Gamma_I(M)$ the {\em $I$-torsion submodule} of $M$.
Recall that $\Gamma_I(M)$ is by definition the subset of $M$ consisting of all elements that are annihilated by some power of $I$, and the assignment $M\mapsto\Gamma_I(M)$ defines a left exact additive covariant functor $\Gamma_I:\mod R\to\mod R$.

\begin{lem}\label{16}
Let $I$ be an ideal of $R$.
Let $C,M$ be $R$-modules and $n\ge1$ an integer.
If $M$ belongs to $[C]_n$, then one has
$
\Ann_R\Gamma_I(M)\supseteq(\Ann_R\Gamma_I(R)\cdot\Ann_R\Gamma_I(C))^n.
$
\end{lem}

\begin{proof}
Let us prove the lemma by induction on $n$.

When $n=1$, the module $M$ is isomorphic to a direct summand of $(R\oplus\bigoplus_{i=0}^a\syz^iC)^{\oplus b}$ for some $a,b\ge0$.
Hence $\Gamma_I(M)$ is isomorphic to a direct summand of $(\Gamma_I(R)\oplus\bigoplus_{i=0}^a\Gamma_I(\syz^iC))^{\oplus b}$.
For $i\ge1$, the syzygy $\syz^iC$ is a submodule of some free module $R^{\oplus c_i}$, and $\Gamma_I(\syz^iC)$ is a submodule of $\Gamma_I(R)^{\oplus c_i}$, which implies that $\Ann_R\Gamma_I(\syz^iC)$ contains $\Ann_R\Gamma_I(R)$.
Hence we obtain
\begin{align*}
\Ann_R\Gamma_I(M)
& \supseteq \Ann_R\Gamma_I(R)\cap(\textstyle\bigcap_{i=0}^a\Ann_R\Gamma_I(\syz^iC)) \\
& = \Ann_R\Gamma_I(R)\cap\Ann_R\Gamma_I(C)
\supseteq \Ann_R\Gamma_I(R)\cdot\Ann_R\Gamma_I(C).
\end{align*}

Let $n\ge2$.
Then $M$ is in $[C]_n=[C]_{n-1}\bullet[C]$, and Proposition \ref{1}(1) says that there is an exact sequence $0 \to X \to Y \to Z\to 0$ with $X\in[C]_{n-1}$ and $Z\in[C]$ such that $M$ is a direct summand of $Y$.
We have an exact sequence $0 \to \Gamma_I(X) \to \Gamma_I(Y) \to \Gamma_I(Z)$, and therefore we obtain
\begin{align*}
\Ann_R\Gamma_I(M)
& \supseteq \Ann_R\Gamma_I(Y)
\supseteq \Ann_R\Gamma_I(X)\cdot\Ann_R\Gamma_I(Z) \\
& \supseteq (\Ann_R\Gamma_I(R)\cdot\Ann_R\Gamma_I(C))^{n-1}\cdot(\Ann_R\Gamma_I(R)\cdot\Ann_R\Gamma_I(C)) \\
& = (\Ann_R\Gamma_I(R)\cdot\Ann_R\Gamma_I(C))^n,
\end{align*}
which is what we want.
\end{proof}

Let $(R,\m)$ be a local ring, and let $M$ be an $R$-module.
We denote by $\LL(M)$ the {\em Loewy length} of $M$, which is by definition the infimum of the intergers $n\ge0$ such that $\m^nM=0$.
Obviously, $\LL(M)$ is finite if and only if $M$ has finite length.
There is a relationship between finite radius and Loewy length:

\begin{prop}\label{15}
Let $(R,\m)$ be a local ring, and let $\X$ be a resolving subcategory of $\mod R$.
If $\radi\X<\infty$, then $\sup_{X\in\X}\{\LL(\Gamma_\m(X))\}<\infty$.
\end{prop}

\begin{proof}
Put $r=\radi\X$.
By definition, there exists an $R$-module $C$ such that $[C]_{r+1}$ contains $\X$.
Let $X$ be a module in $\X$.
It follows from Lemma \ref{16} that the annihilator $\Ann\Gamma_\m(X)$ contains the ideal $(\Ann\Gamma_\m(R)\cdot\Ann\Gamma_\m(C))^{r+1}$.
As $\Gamma_\m(R)$ and $\Gamma_\m(C)$ have finite length, they have finite Loewy length.
Set $a=\LL(\Gamma_\m(R))$ and $b=\LL(\Gamma_\m(C))$.
Then $\Ann\Gamma_\m(X)$ contains $\m^{(a+b)(r+1)}$, which means that $\LL(\Gamma_\m(X))$ is at most $(a+b)(r+1)$.
Since the number $(a+b)(r+1)$ is independent of the choice of $X$, we have $\sup_{X\in\X}\{\LL(\Gamma_\m(X))\}\le(a+b)(r+1)<\infty$.
\end{proof}

The following is the essential part of Theorem \ref{3}(1).

\begin{thm}\label{t1}
Let $(R,\m,k)$ be a local ring of positive dimension.
Let $\X$ be a resolving subcategory of $\mod R$.
If $\X$ contains $k$, then $\X$ has infinite radius.
\end{thm}

\begin{proof}
We claim that $R/\m^i$ belongs to $\X$ for all integers $i>0$.
Indeed, there is an exact sequence
$
0 \to \m^{i-1}/\m^i \to R/\m^i \to R/\m^{i-1} \to 0,
$
and the left term belongs to $\X$ as it is a $k$-vector space.
Induction on $i$ shows the claim.

We have $\LL(\Gamma_\m(R/\m^i))=\LL(R/\m^i)=i$, where the second equality follows from the assumption that $\dim R>0$.
Therefore it holds that
$
\sup_{X\in\X}\{\LL(\Gamma_\m(X))\}\ge\sup_{i>0}\{\LL(\Gamma_\m(R/\m^i))\}=\sup_{i>0}\{ i\}=\infty,
$
which implies $\radi\X=\infty$ by Proposition \ref{15}.
\end{proof}

Now, let us prove Theorem \ref{3}(1).
By Remark \ref{rrr}, we may assume that $(R,\p)$ is a local ring with $\dim R>0$ and that $M$ is a nonzero $R$-module in $\X$.
The assumption $\p M=0$ says that $M$ is a nonzero $k$-vector space, where $k=R/\p$ is the residue field of $R$.
Since $\X$ is closed under direct summands, $k$ belongs to $\X$.
Theorem \ref{t1} yields $\radi\X=\infty$.

\subsection{Proof of Theorem \ref{3}(2)(3)}

Establishing several preliminary lemmas and propositions are necessary, which will also be used in the proof of Theorem \ref{3}(4).

We begin with stating an elementary lemma, whose proof we omit.

\begin{lem}\label{a}
Let $(R,\m,k)$ be a local ring.
Let $0 \to L \to M \to N \to 0$ be an exact sequence of $R$-modules.
Then one has
$
\inf\{\depth L,\depth N\}=\inf\{\depth M,\depth N\}.
$
\end{lem}

For an ideal $I$ of $R$, we denote by $V(I)$ (respectively, $D(I)$) the closed (respectively, open) subset of $\Spec R$ defined by $I$ in the Zariski topology, namely, $V(I)$ is the set of prime ideals containing $I$ and $D(I)=\Spec R\setminus V(I)$.
For an $R$-module $M$ we denote by $\NF(M)$ the {\em nonfree locus} of $M$, that is, the set of prime ideals $\p$ of $R$ such that the $R_\p$-module $M_\p$ is nonfree.
It is well-known that $\NF(M)$ is a closed subset of $\Spec R$.

The next result builds, out of each module in a resolving subcategory and each point in its nonfree locus, another module in the same resolving subcategory whose nonfree locus coincides with the closure of the point.
Such a construction has already been given in \cite[Theorem 4.3]{res}, but we need in this paper a more detailed version.
Indeed, the following lemma yields a generalization of \cite[Theorem 4.3]{res}.

\begin{lem}\label{c}
Let $M$ be an $R$-module.
For every $\p\in\NF(M)$ there exists $X\in\res M$ satisfying $\NF(X)=V(\p)$ and $\depth X_\q=\inf\{\depth M_\q,\depth R_\q\}$ for all $\q\in V(\p)$.
\end{lem}

\begin{proof}
Note that $V(\p)$ is contained in $\NF(M)$.
If $V(\p)=\NF(M)$, then we can take $X:=M\oplus R$.
Suppose $V(\p)$ is strictly contained in $\NF(M)$.
Then there is a prime ideal $\r$ in $\NF(M)$ that is not in $V(\p)$.
Choose an element $x\in\p\setminus\r$.
By \cite[Proposition 4.2]{res}, we have a commutative diagram with exact rows
$$
\begin{CD}
0 @>>> \syz M @>>> F @>>> M @>>> 0 \\
@. @V{x}VV @VVV @| \\
0 @>>> \syz M @>>> N @>>> M @>>> 0 
\end{CD}
$$
where $F$ is free, $V(\p)\subseteq\NF(N)\subseteq\NF(M)$ and $D((x))\cap\NF(N)=\emptyset$.
The second row shows that $N$ belongs to $\res M$.
Since $\r$ is in $D((x))$, it is not in $\NF(N)$, and we have
$
V(\p)\subseteq\NF(N)\subsetneq\NF(M).
$

Now we claim that $\depth N_\q=\inf\{\depth M_\q,\depth R_\q\}$ for all $\q\in V(\p)$.
Indeed, localizing the above diagram at $\q$ and taking long exact sequences with respect to Ext, we get a commutative diagram with exact rows
$$
\begin{CD}
E^{i-1}(M_\q) @>>> E^i((\syz M)_\q) @>>> E^i(F_\q) @>>> E^i(M_\q) @>>> E^{i+1}((\syz M)_\q) \\
@| @V{x}V{(1)}V @VVV @| @V{x}V{(2)}V \\
E^{i-1}(M_\q) @>{(3)}>> E^i((\syz M)_\q) @>>> E^i(N_\q) @>>> E^i(M_\q) @>{(4)}>> E^{i+1}((\syz M)_\q)
\end{CD}
$$
for $i\in\ZZ$, where $E^i(-)=\Ext_{R_\q}^i(\kappa(\q),-)$.
As $x$ is an element of $\q$, the maps (1),(2) are zero maps, and so are (3),(4).
Thus we have a short exact sequence
$$
0 \to \Ext^i(\kappa(\q),(\syz M)_\q) \to \Ext^i(\kappa(\q),N_\q) \to \Ext^i(\kappa(\q),M_\q) \to 0
$$
for each integer $i$.
It is easy to see from this that the first equality in the following holds, while the second equality is obtained by applying Lemma \ref{a} to the exact sequence $0\to (\syz M)_\q\to F_\q\to M_\q\to 0$.
$$
\depth N_\q=\inf\{\depth (\syz M)_\q,\depth M_\q\}=\inf\{\depth R_\q,\depth M_\q\}.
$$
Thus the claim follows.

If $V(\p)=\NF(N)$, then we can take $X:=N$.
If $V(\p)$ is strictly contained in $\NF(N)$, then the above procedure gives rise to an $R$-module $L\in\res N$ (hence $L\in\res M$) with
$
V(\p)\subseteq\NF(L)\subsetneq\NF(N)\subsetneq\NF(M)
$
such that
$
\depth L_\q=\inf\{\depth N_\q,\depth R_\q\}=\inf\{\depth M_\q,\depth R_\q\}
$
for all $\q\in V(\p)$.
Since $\Spec R$ is a Noetherian space, iteration of this procedure must stop in finitely many steps, and we eventually obtain such a module $X$ as in the lemma.
\end{proof}

The next lemma will play a crucial role in the proofs of our theorems.
The main idea of the proof is similar to that of Lemma \ref{16}, but a much closer examination is necessary to be made.

\begin{lem}\label{b}
Let $S\to R$ be a homomorphism of rings.
Let $C,M$ be $R$-modules with $M\in[C]_n^R$, and let $N$ be an $S$-module of injective dimension $m<\infty$.
Then one has
\begin{align*}
\textstyle\bigcap_{i>0}\Ann_S\Ext_S^i(M,N) & =\textstyle\bigcap_{i=1}^m\Ann_S\Ext_S^i(M,N)\\
& \supseteq(\textstyle\prod_{i=1}^m\Ann_S\Ext_S^i(R,N)\cdot\Ann_S\Ext_S^i(C,N))^n.
\end{align*}
\end{lem}

\begin{proof}
For each integer $i\le 1$ and  $R$-module $L$, set $\a_L^i=\Ann_S\Ext_S^i(L,N)$.
Note that $\a_L^h=S$ for all $h>m$ since $\Ext_S^h(L,N)=0$.
It suffices to prove that
$
\a_M^i\supseteq(\prod_{j=i}^m\a_R^j\a_C^j)^n.
$
Let us proceed by induction on $n$.

When $n=0$, we have $M=0$, and the above two ideals coincide with $S$.

Let $n=1$.
Then $M$ is isomorphic to a direct summand of a finite direct sum of copies of $R\oplus(\bigoplus_{j=0}^l\syz^jC)$.
Hence $\Ext_S^i(M,N)$ is isomorphic to a direct summand of a finite direct sum of copies of $\Ext_S^i(R,N)\oplus(\bigoplus_{j=0}^l\Ext_S^i(\syz^jC,N))$.
Thus we have
$
\a_M^i\supseteq\a_R^i\cap\left(\bigcap_{j=0}^l\a_{\syz^jC}^i\right).
$
For each $j\ge1$ there is an exact sequence $0 \to \syz^jC \to R^{\oplus k_j} \to \syz^{j-1}C \to 0$, which induces an exact sequence $\Ext_S^i(R,N)^{\oplus k_j} \to \Ext_S^i(\syz^jC,N) \to \Ext_S^{i+1}(\syz^{j-1}C,N)$.
This gives
$$
\a_{\syz^jC}^i
\supseteq \a_R^i\a_{\syz^{j-1}C}^{i+1}
\supseteq \a_R^i\a_R^{i+1}\a_{\syz^{j-2}C}^{i+2}
\supseteq \cdots
\supseteq \a_R^i\a_R^{i+1}\cdots\a_R^{i+j-1}\a_C^{i+j}.
$$
Regarding $\a_R^i\a_R^{i+1}\cdots\a_R^{i+j-1}\a_C^{i+j}$ as $\a_C^i$ when $j=0$, we have $\a_{\syz^jC}^i\supseteq\a_R^i\a_R^{i+1}\cdots\a_R^{i+j-1}\a_C^{i+j}$ for all $j\ge0$.
Thus we obtain:
\begin{align*}
\a_M^i
\supseteq \a_R^i\cap(\textstyle\bigcap_{j=0}^l\a_R^i\a_R^{i+1}\cdots\a_R^{i+j-1}\a_C^{i+j})
& \supseteq (\a_R^i\a_R^{i+1}\cdots\a_R^m\a_R^{m+1}\cdots)(\a_C^i\a_C^{i+1}\cdots\a_C^m\a_C^{m+1}\cdots) \\
& = (\a_R^i\a_R^{i+1}\cdots\a_R^m)(\a_C^i\a_C^{i+1}\cdots\a_C^m)
\supseteq \textstyle\prod_{j=i}^m\a_R^j\a_C^j.
\end{align*}

Now let us consider the case where $n\ge2$.
We have $M\in[C]_n=[C]_{n-1}\bullet[C]$, and by Proposition \ref{1}(1), there exists an exact sequence $0 \to X \to Y \to Z\to 0$ with $X\in[C]_{n-1}$ and $Z\in[C]$ such that $M$ is a direct summand of $Y$.
Using the induction hypothesis, we have
$$
\a_M^i
\supseteq \a_Y^i
\supseteq \a_X^i\cdot\a_Z^i
\supseteq (\textstyle\prod_{j=i}^m\a_R^j\a_C^j)^{n-1}\cdot(\textstyle\prod_{j=i}^m\a_R^j\a_C^j)
= (\textstyle\prod_{j=i}^m\a_R^j\a_C^j)^n,
$$
which completes the proof of the lemma.
\end{proof}

Here we prepare a lemma, which is an easy consequence of Krull's intersection theorem.

\begin{lem}\label{e}
Let $(R,\m)$ be a local ring and $M$ an $R$-module.
Then $\Ann_RM=\bigcap_{i>0}\Ann_R(M/\m^iM)$.
\end{lem}

Now we can prove the following proposition, which will be the base of the proofs of our theorems.
Actually, all of them will be proved by making use of this proposition.

\begin{prop}\label{14}
Let $(R,\m,k)$ be a local ring.
Let $\X$ be a resolving subcategory of $\mod R$.
If $\X$ contains a module $M$ such that $0< \pd_RM<\infty$, then one has $\radi \X =\infty$.
\end{prop}

\begin{proof}
Applying Lemma \ref{c} to $\m\in\NF(M)$, we find a module $X\in\res M\subseteq\X$ satisfying $\NF(X)=\{\m\}$ and $\depth X=\inf\{\depth M,\depth R\}$.
Since $M$ has finite projective dimension, the depth of $M$ is at most that of $R$.
Hence we have $\depth X=\depth M$.
Note that the subcategory of $\mod R$ consisting of $R$-modules of finite projective dimension is resolving.
Since it contains $M$, it also contains $\res M$.
This implies that $X$ has finite projective dimension, and we have $\pd_RX=\depth R-\depth X=\depth R-\depth M=\pd_RM$.
Thus, replacing $M$ with $X$, we may assume that $M$ is locally free on the punctured spectrum of $R$.
Taking the $n$-th syzygy of $M$ where $n=\pd_RM-1\ge0$, we may also assume that the projective dimension of $M$ is equal to $1$.

Now $\Ext_R^1(M,R)$ is a nonzero $R$-module of finite length, and we can choose a socle element $0\ne\sigma\in\Ext_R^1(M,R)$.
It can be represented as a short exact sequence:
$$
\sigma: 0 \to R \to N \to M \to 0.
$$
The module $N$ belongs to $\X$, is locally free on the punctured spectrum of $R$ and has projective dimension at most $1$.
Hence $\pd_RN=1$ if and only if $\Ext_R^1(N,R)\ne0$.
Applying the functor $\Hom_R(-,R)$, we get an exact sequence $R \xrightarrow{f} \Ext_R^1(M,R) \to \Ext_R^1(N,R) \to 0$, where $f$ sends $1\in R$ to $\sigma\in\Ext_R^1(M,R)$.
Hence we obtain an exact sequence
$$
0 \to k \to \Ext_R^1(M,R) \to \Ext_R^1(N,R) \to 0.
$$
This implies $\len(\Ext_R^1(N,R))=\len(\Ext_R^1(M,R))-1$.
Replacing $M$ by $N$ and repeating this process if $\len(\Ext_R^1(N,R))>0$, we can assume that $\Ext_R^1(N,R)=0$.
Therefore $\Ext_R^1(M,R)\cong k$.
Since $\pd_RM=1$, we easily get an isomorphism $\tr M\cong k$.
Taking the transpose of this isomorphism, we see that $\tr k$ is isomorphic to $M$ up to free summand (cf. \cite[Proposition (2.6)(d)]{AB}).
It follows that $\tr k$ belongs to $\X$.

We claim that $\tr L$ is in $\X$ for any $R$-module $L$ of finite length.
This is shown by induction on $\len L$.
If $\len L>0$, then there is an exact sequence $0 \to L' \to L \to k \to 0$, and applying \cite[Lemma (3.9)]{AB} (see also \cite[Proposition 3.3(3)]{crs}), we have an exact sequence
$$
0=(L')^\ast \to \tr k \to \tr L\oplus R^{\oplus n} \to \tr L' \to 0,
$$
where the equality follows from the fact that $R$ has positive depth.
(As $R$ possesses a module of finite positive projective dimension, the depth of $R$ is positive.)
The induction hypothesis implies $\tr L'\in\X$, and the above exact sequence shows $\tr L\in\X$, as desired.

Now, assume that we have $\radi\X=r<\infty$.
We want to deduce a contradiction.
There is a ball $[C]_{r+1}^R$ that contains $\X$.
Since $\X$ contains $\tr_R(R/\m^i)$ for all $i>0$, the ball $[C]_{r+1}^R$ also contains it.
Taking the completions, we have $\tr_{\widehat R}(\widehat R/\m^i\widehat R)\in[\widehat C]_{r+1}^{\widehat R}$ for all $i>0$.
By virtue of Cohen's structure theorem, there exists a surjective homomorphism $S\to\widehat R$ such that $S$ is a Gorenstein local ring with $\dim S=\dim R=:d$.
Let $\n$ denote the maximal ideal of $S$ and note that we have $\widehat R/\m^i\widehat R=\widehat R/\n^i\widehat R$ for any $i>0$.
Lemma \ref{b} gives an inclusion relation
$$
\textstyle\bigcap_{i>0}\Ann_S\Ext_S^i(\tr_{\widehat R}(\widehat R/\n^i\widehat R),S)\supseteq(\textstyle\prod_{i=1}^d\Ann_S\Ext_S^i(\widehat R,S)\cdot\Ann_S\Ext_S^i(\widehat C,S))^{r+1}.
$$
Fix an integer $i>0$ and let $a_1,\dots,a_m$ be a system of generators of the ideal $\n^i$ of $S$.
There is an exact sequence
$
{\widehat R}^{\oplus m} \xrightarrow{(a_1,\dots,a_m)} \widehat R \to \widehat R/\n^i\widehat R \to 0
$
of $\widehat R$-modules.
Dualizing this by $\widehat R$ induces an exact sequence
$$
0=\Hom_{\widehat R}(\widehat R/\n^i\widehat R,\widehat R) \to \widehat R \xrightarrow{\left(\begin{smallmatrix}
a_1\\
\vdots\\
a_m
\end{smallmatrix}\right)} {\widehat R}^{\oplus m} \to \tr_{\widehat R}(\widehat R/\n^i\widehat R) \to 0,
$$
where the equality follows since $\depth\widehat R=\depth R>0$.
This makes an exact sequence
$$
\Hom_S(\widehat R,S)^{\oplus m} \xrightarrow{(a_1,\dots,a_m)} \Hom_S(\widehat R,S) \to \Ext_S^1(\tr_{\widehat R}(\widehat R/\n^i\widehat R),S),
$$
which yields an injection
$
\Hom_S(\widehat R,S)/\n^i\Hom_S(\widehat R,S)\to\Ext_S^1(\tr_{\widehat R}(\widehat R/\n^i\widehat R),S).
$
Thus
\begin{align*}
& \Ann_S(\Hom_S(\widehat R,S)/\n^i\Hom_S(\widehat R,S)) \supseteq\Ann_S(\Ext_S^1(\tr_{\widehat R}(\widehat R/\n^i\widehat R),S))\\
& \supseteq\textstyle\bigcap_{i>0}\Ann_S\Ext_S^i(\tr_{\widehat R}(\widehat R/\n^i\widehat R),S)
\supseteq(\textstyle\prod_{i=1}^d\Ann_S\Ext_S^i(\widehat R,S)\cdot\Ann_S\Ext_S^i(\widehat C,S))^{r+1},
\end{align*}
and we obtain
\begin{align*}
& (\textstyle\prod_{i=1}^d\Ann_S\Ext_S^i(\widehat R,S)\cdot\Ann_S\Ext_S^i(\widehat C,S))^{r+1}\\
& \subseteq \textstyle\bigcap_{i>0}\Ann_S(\Hom_S(\widehat R,S)/\n^i\Hom_S(\widehat R,S))= \Ann_S\Hom_S(\widehat R,S),
\end{align*}
where the equality follows from Lemma \ref{e}.

Let $I$ be the kernel of the surjection $S\to\widehat R$.
Since $\dim S=d=\dim\widehat R$, the ideal $I$ of $S$ has height zero.
Hence there exists a minimal prime ideal $\p$ of $S$ which contains $I$.
Since we have a ring epimorphism from the Artinian Gorenstein local ring $S_\p$ to ${\widehat R}_\p$, the ${\widehat R}_\p$-module $\Hom_S(\widehat R,S)_\p=\Hom_{S_\p}({\widehat R}_\p,S_\p)$ is isomorphic to the injective hull of the residue field of ${\widehat R}_\p$, which is in particular nonzero.
This implies that $\p$ contains the ideal $\Ann_S\Hom_S(\widehat R,S)$.
Therefore, for some integer $1\le l\le d$ the ideal $\p$ contains either $\Ann_S\Ext_S^l(\widehat R,S)$ or $\Ann_S\Ext_S^l(\widehat C,S)$.
If $\p$ contains $\Ann_S\Ext_S^l(\widehat R,S)$, then we have $\Ext_{S_\p}^l({\widehat R}_\p,S_\p)\ne0$, which contradicts the fact that $S_\p$ is injective as an $S_\p$-module.
Similarly, we have a contradiction when $\p$ contains $\Ann_S\Ext_S^l(\widehat C,S)$.
This contradiction proves that $\radi\X=\infty$.
\end{proof}

Now we can show the essential part of Theorem \ref{3}(2)(3).

\begin{thm}\label{t23}
Let $R$ be a local ring with $\dim R>0$.
Let $\X$ be a resolving subcategory of $\mod R$.
One has $\radi\X=\infty$ if there exists a module $M\in\X$ with $0<\Gdim_RM<\infty$ that satisfies either of the following conditions.
\begin{enumerate}[\rm(1)]
\item
$\syz^{-2}\syz^gM\in\X$, where $g=\Gdim_RM$.
\item
$\res(\syz^nM)$ is a thick subcategory of $\TR(R)$ for some $n\ge0$.
\end{enumerate}
\end{thm}

\begin{proof}
According to Proposition \ref{14}, it suffices to show that $\X$ contains a module of projective dimension one.

(1) We consider a construction whose idea essentially comes from the Auslander-Buchweitz approximation theorem \cite[Theorem 1.1]{AB''}.
There are exact sequences $0\to\syz^gM\to R^{\oplus a}\to\syz^{g-1}M\to0$ and $0\to\syz^gM\to R^{\oplus b}\to\syz^{-1}\syz^gM\to0$, where the latter is possible as $\syz^gM$ is a totally reflexive module.
We make the following pushout diagram.
$$
\begin{CD}
@. 0 @. 0 @. \\
@. @VVV  @VVV @. \\
0 @>>> \syz^gM @>>> R^{\oplus a} @>>> \syz^{g-1}M @>>> 0 \\
@. @VVV @VVV @| \\
0 @>>> R^{\oplus b} @>>> N @>>> \syz^{g-1}M @>>> 0 \\
@. @VVV  @VVV @. \\
@. \syz^{-1}\syz^gM @= \syz^{-1}\syz^gM \\
@. @VVV  @VVV @. \\
 @. 0 @. 0\\
\end{CD}
$$
As $\syz^{-1}\syz^gM$ is totally reflexive, we have $\Ext_R^1(\syz^{-1}\syz^gM,R)=0$.
Hence the second column in the above diagram splits, and we get an exact sequence
$
0\to R^{\oplus b}\to R^{\oplus a}\oplus\syz^{-1}\syz^gM\to\syz^{g-1}M\to0.
$
There is an exact sequence $0\to\syz^{-1}\syz^gM\to R^{\oplus c}\to\syz^{-2}\syz^gM\to0$, and taking the direct sum with $0\to R^{\oplus a}\xrightarrow{=}R^{\oplus a}\to0\to0$, we have an exact sequence
$
0\to R^{\oplus a}\oplus\syz^{-1}\syz^gM\to R^{\oplus(a+c)}\to\syz^{-2}\syz^gM\to0.
$
Thus the following pushout diagram is obtained.
$$
\begin{CD}
@. @. 0 @. 0 \\
@. @. @VVV @VVV \\
0 @>>> R^{\oplus b} @>>> R^{\oplus a}\oplus\syz^{-1}\syz^gM @>>> \syz^{g-1}M @>>> 0 \\
@. @| @VVV @VVV \\
0 @>>>R^{\oplus b}@>>> R^{\oplus(a+c)} @>>> L @>>> 0 \\
@. @. @VVV @VVV \\
@. @. \syz^{-2}\syz^gM @= \syz^{-2}\syz^gM \\
@. @. @VVV @VVV \\
@. @. 0 @. 0\\
\end{CD}
$$
As $\syz^{g-1}M$ and $\syz^{-2}\syz^gM$ are in $\X$, the module $L$ is also in $\X$.
The second row shows that $L$ has projective dimension at most $1$.
Since $\syz^{-2}\syz^gM$ is totally reflexive but $\syz^{g-1}M$ is not, it follows from the second column that $L$ is nonfree.
Therefore the projective dimension of $L$ is equal to $1$.

(2) Since by assumption $\res\syz^nM$ is a subcategory of $\TR(R)$, the module $\syz^nM$ is totally reflexive.
Hence $n\ge g:=\Gdim_RM$.

We claim that $\res\syz^nM=\res\syz^gM$.
In fact, since $n-g\ge0$ and $\syz^nM=\syz^{n-g}(\syz^gM)$, we observe that $\res\syz^nM$ is contained in $\res\syz^gM$.
There is an exact sequence
$$
0 \to \syz^nM \to F_{n-1} \to F_{n-2} \to \cdots \to F_{g+1} \to F_g \to \syz^gM \to 0
$$
of totally reflexive $R$-modules with each $F_i$ being free.
Since $\res\syz^nM$ is assumed to be thick in $\TR(R)$, decomposing the above exact sequence into short exact sequences of totally reflexive modules, we see that $\syz^gM$ belongs to $\res\syz^nM$.
Therefore $\res\syz^gM$ is contained in $\res\syz^nM$.
Thus the claim follows.

Set $\X=\res\syz^nM=\res\syz^gM$.
Our assumption implies that $\X$ is closed under cosyzygies, whence $\syz^{-2}\syz^gM\in\X$.
By (1), we conclude that $\X$ contains a module of projective dimension $1$.
\end{proof}

Now, using Remark \ref{rrr} and Theorem \ref{t23}, we deduce Theorem \ref{3}(2)(3).

\subsection{Proof of Theorem \ref{3}(4)}\label{cmplxt}

We use the notion of a module of reducible complexity, which has been introduced by Bergh \cite{Be}.
Let us recall the definition.

\begin{dfn} 
The subcategory $\C^r_R$ of $\mod R$ is defined inductively as follows.
\begin{enumerate}[(1)]
\item
Every module of finite projective dimension belongs to $\C^r_R$.
\item
A module $M$ with $0<\cx_RM<\infty$ belongs to $\C^r_R$ if there exists a homogeneous element $\eta\in\Ext_R^*(M,M)$ with $|\eta|>0$ which is represented by a short exact sequence
$
\ses{M}{K}{\syz^{|\eta|-1} M}
$ 
with $K\in\C^r_R$, $\cx K<\cx M$ and $\depth K= \depth M$.
\end{enumerate}
An $R$-module is said to have {\em reducible complexity} if it is in $\C^r_R$.
\end{dfn} 

The result below is shown in \cite[Proposition 2.2(i)]{Be}, which is implicitly stated in \cite{AGP}.

\begin{prop}\label{bagp}
Let $R$ be a local ring.
Every $R$-module of finite CI-dimension has reducible complexity.
\end{prop}

In a resolving subcategory, for any fixed integer $n\ge0$, existence of modules of CI-dimension $n$ is equivalent to existence of modules of projective dimension $n$.

\begin{lem}\label{main3}
Let $R$ be a local ring.
Let $\X$ be a resolving subcategory of $\mod R$.
Suppose that there is a module $M\in\X$ such that $\CIdim_RM<\infty$.
Then $\X$ contains a module $N$ with $\pd_RN=\CIdim_RM$.
\end{lem}

\begin{proof}
Since $M$ has finite CI-dimension, it has finite complexity.
It follows from Proposition \ref{bagp} that $M$ has reducible complexity.
If $\cx M=0$, then $\pd M<\infty$, and we can take $N:=M$.
Hence we may assume $\cx M>0$.
There exists an exact sequence
$
\ses{M}{K}{\syz_R^{|\eta|-1} M}
$ 
with $\cx K<\cx M$ and $\depth K= \depth M$, where $\eta$ is a homogeneous element of $\Ext_R^*(M,M)$.
We have $K\in\X$ and $\CIdim K=\depth R-\depth K=\depth R-\depth M=\CIdim M$.
Replacing $M$ with $K$ and iterating this procedure, we can eventually arrive at a module $N\in \X$ with $\CIdim N = \CIdim M$ and $\cx N=0$.
The module $N$ has finite projective dimension, and we have $\pd N=\CIdim N=\CIdim M$.
\end{proof}

Lemma \ref{main3} and Proposition \ref{14} immediately yield the following theorem.
This is not only the essential part of Theorem \ref{3}(4) but also a generalization of Proposition \ref{14}.

\begin{thm}\label{t4}
Let $R$ be a local ring with $\dim R>0$.
Let $\X$ be a resolving subcategory of $\mod R$.
Suppose that there exists a module $M\in\X$ with $0<\CIdim_RM<\infty$.
Then the radius of $\X$ is infinite.
\end{thm}

Theorem \ref{3}(4) now follows from Theorem \ref{t4} and Remark \ref{rrr}.

\subsection{Another proof of Theorem \ref{t4}}

In the next theorem, we study the thickness of resolving subcategories of modules of CI-dimension at most zero.
This will give another proof of Theorem \ref{t4}.

\begin{thm}\label{4}
Let $R$ be a local ring.
\begin{enumerate}[\rm(1)]
\item
Let $M$ be an $R$-module of CI-dimension at most zero.
Then $\syz^{-1}M$ belongs to $\res M$.
\item
Let $\X$ be a resolving subcategory of $\mod R$.
Suppose that every module in $\X$ has CI-dimension at most zero.
Then $\X$ is a thick subcategory of $\TR(R)$.
\end{enumerate}
\end{thm} 

\begin{proof}
(1) Proposition \ref{bagp} implies that $M$ has reducible complexity.
Let $K_0=M$ and let $K_{i+1}$ be a reduction in complexity of $K_i$ for each $i\ge0$.
Then we have a short exact sequence
$
0 \to K_i \xrightarrow{f_i} K_{i+1} \to \syz^{t_i-1}K_i \to 0
$
with $t_i>0$ (see also \cite[Proposition 7.2]{AGP}), and eventually we must have $\cx_R K_e=0$ for some $e\ge0$.
Then $K_e$ has finite projective dimension.
As $\CIdim_RM=0$, we have $\depth K_e=\depth M=\depth R$.
Therefore $K_e$ is a free module.
Note that the above exact sequence also shows that
\begin{equation}\label{953}
K_i\in\res M\text{ for all }i\ge0.
\end{equation}
If $e=0$, then $M$ is free and we have $\syz^{-1}M=0\in\res M$.
So we may assume $e\ge1$.

We claim that for each $0\le i\le e-1$ the cokernel $C_i$ of the composite map $f_i\cdots f_1f_0: M\to K_{i+1}$ belongs to $\res M$.
Let us show this claim by induction on $i$.
When $i=0$, we have $C_i=\syz^{t_0-1}M\in\res M$.
Let $i\ge1$.
We have the following commutative diagram with exact rows and columns.
$$
\begin{CD}
@. @. 0 @. 0 \\
@. @. @VVV @VVV \\
0 @>>> M @>{f_{i-1}\cdots f_0}>> K_i @>>> C_{i-1} @>>> 0 \\
@. @| @V{f_i}VV @VVV \\
0 @>>> M @>>> K_{i+1} @>>> C_i @>>> 0 \\
@. @. @VVV @VVV \\
@. @. \syz^{t_i-1}K_i @= \syz^{t_i-1}K_i \\
@. @. @VVV @VVV \\
@. @. 0 @. 0
\end{CD}
$$
The induction hypothesis implies that $C_{i-1}$ belongs to $\res M$.
Since $\syz^{t_i-1}K_i$ is in $\res M$ by \eqref{953}, the right column shows that $C_i$ is also in $\res M$.
Thus the claim follows.

Now we have a short exact sequence
$
0 \to M \xrightarrow{f_{e-1}\cdots f_1f_0} K_e \to C_{e-1} \to 0,
$
where $K_e$ is free and $C_{e-1}$ is in $\res M$ by the claim.
Since $M$ is totally reflexive and $\TR(R)$ is a resolving subcategory of $\mod R$, all the modules in $\res M$ are totally reflexive.
Hence all modules appearing in the above exact sequence belong to $\TR(R)$.
It is easy to verify that there exists an isomorphism $C_{e-1}\cong\syz^{-1}M\oplus F$ with $F$ being free.
Consequently, $\syz^{-1}M$ belongs to $\res M$.

(2) By assumption, $\X$ is a subcategory of $\TR(R)$.
Thanks to Proposition \ref{12}, it is enough to show that $\X$ is closed under cosyzygies.
Let $M$ be an $R$-module in $\X$.
Then $M$ is of CI-dimension at most zero, and $\syz^{-1}M$ belongs to $\res M$ by (1).
Since $\X$ is resolving and contains $M$, it also contains $\res M$.
Thus $\syz^{-1}M$ is in $\X$.
\end{proof}

Theorem \ref{4}(2) and Proposition \ref{12} immediately implies:

\begin{cor}\label{1559}
Let $R$ be a local complete intersection.
The following two are the same.
\begin{itemize}
\item
A resolving subcategory of $\mod R$ contained in $\CM(R)$.
\item
A thick subcategory of $\CM(R)$ containing $R$.
\end{itemize}
\end{cor}

Now let us give another proof of Theorem \ref{t4}.
Let $R$ be a local ring with $\dim R>0$, and let $M\in\X$ be an $R$-module with $0<\CIdim_RM<\infty$.
Then we have $c:=\CIdim_RM=\Gdim_RM$, and $\syz^cM$ has CI-dimension zero.
In particular, $\syz^cM$ is totally reflexive, and hence so is $\syz^{-1}\syz^cM$.
We have
$
0=\CIdim_R(\syz^cM)=\sup\{\CIdim_R(\syz^{-1}\syz^cM)-1,0\},
$
which especially says that $\syz^{-1}\syz^cM$ has finite CI-dimension.
Therefore $\CIdim_R(\syz^{-1}\syz^cM)=\Gdim_R(\syz^{-1}\syz^cM)=0$, and Theorem \ref{4}(1) yields $\syz^{-2}\syz^cM\in\res(\syz^cM)\subseteq\X$.
Now Theorem \ref{t23}(1) implies that the radius of $\X$ is infinite.


\section{Proof of Theorem II}\label{section5}

In this section we prove the main Theorem II from Introduction. In fact, we can prove significantly more general statements (Theorems \ref{88} and \ref{13}). In order to state and prove such results we need to first introduce a couple of definitions related to the concept of radius.

\begin{dfn}
Let $\X,\Y$ be subcategories of $\mod R$.
We put $|\X|=\add\X$, and set $\X*\Y=||\X|\circ|\Y||$.
(The notation ``$\circ$'' was introduced in Definition \ref{keydef}.)
For an integer $r>0$, set
$$
|\X|_r=
\begin{cases}
|\X| & (r=1),\\
|\X|_{r-1}*\X & (r\ge2).
\end{cases}
$$
\end{dfn}

Let $\X,\Y,\Z$ be subcategories of $\mod R$.
We observe that an object $M\in\mod R$ is in $\X*\Y$ if and only if there is an exact sequence $0 \to X \to E \to Y \to 0$ with $X\in|\X|$ and $Y\in|\Y|$ such that $M$ is a direct summand of $E$.
Also, one has $(\X*\Y)*\Z=\X*(\Y*\Z)$ and $|\X|_a*|\X|_b=|\X|_{a+b}$ for all $a,b>0$.

\begin{dfn}
For a subcategory $\X$ of $\mod $ we define $\size\X$ (respectively, $\rank\X$) to be the infimum of integers $n\ge0$ such that $\X\subseteq|G|_{n+1}$ (respectively, $\X=|G|_{n+1}$) for some $G\in\mod R$.
\end{dfn}

It always holds that $\size\X\le\rank\X$.
Since $|\X|_n\subseteq[\X]_n$ for all $n>0$, one has $\dim\X\ge\radius\X\le\size\X$.
If $\X$ is resolving, then $\dim\X\le\rank\X$.

For an $R$-module $M$, we denote by $M^\oplus$ an object in $\add_RM$.

\begin{prop}\label{6}
Let $I$ be an ideal of $R$ and let $M$ be an $R/I$-module.
\begin{enumerate}[\rm(1)]
\item
There is an exact sequence $0 \to I^\oplus \to \syz_RM \to \syz_{R/I}M \to 0$.
\item
One has $\syz_R^nM\in|\syz_{R/I}^nM\oplus(\bigoplus_{i=0}^{n-1}\syz_R^iI)|_{n+1}$ for all $n\ge0$.
\end{enumerate}
\end{prop}

\begin{proof}
(1) Take a surjection from a free $R$-module $F$ to $M$.
Then this factors through a surjection $F/IF\to M$.
The assertion follows from this.

(2) We induce on $n$.
Let $n>0$.
The induction hypothesis shows $\syz_R^{n-1}\syz_{R/I}M\in|\syz_{R/I}^nM\oplus(\bigoplus_{i=0}^{n-2}\syz_R^iI)|_n$.
By (1) we have an exact sequence $0 \to \syz_R^{n-1}I^\oplus \to \syz_R^nM \to \syz_R^{n-1}\syz_{R/I}M \to 0$.
Now the assertion follows.
\end{proof}

For $n\ge0$ we denote by $\syz_R^n(\mod R)$ the subcategory of $\mod R$ consisting of $n$-th syzygies of $R$-modules.
For an ideal of $R$, let $\syz_R^n(\mod R/I)$ be the subcategory of $\mod R$ consisting of $n$-th syzygies of $R$-modules annihilated by $I$.

\begin{cor}
Let $d=\dim R<\infty$.
Suppose that $R/\p$ is regular for all $\p\in\Min R$.
Then $\size\syz^d(\mod R)<\infty$.
\end{cor}

\begin{proof}
There is a filtration $R=I_0\supsetneq I_1\supsetneq\cdots\supsetneq I_n=0$ of ideals of $R$ such that for each $i$ one has $I_i/I_{i+1}\cong R/\p_i$ with $\p_i\in\spec R$.
Choose a minimal prime $\q_i$ contained in $\p_i$.
Let $M$ be an $R$-module.
Setting $M_i=I_iM/I_{i+1}M$, we have an exact sequence $0\to\syz_R^d(I_{i+1}M)\to\syz_R^d(I_iM)\oplus R^\oplus\to\syz_R^dM_i\to0$.
Note that each $d$-th syzygy $R/\q_i$-module is free.
Hence $\syz_R^dM_i\in|R/\q_i\oplus L_i|_{d+1}$ by Proposition \ref{6}(2), where $L_i:=\bigoplus_{j=0}^{d-1}\syz_R^j\q_i$.
Thus $\syz_R^dM\in|\bigoplus_{i=1}^n(R/\q_i\oplus L_i)|_{n(d+1)}$, which implies $\size\syz^d(\mod R)<n(d+1)<\infty$.
\end{proof}

\begin{cor}\label{7}
Let $I$ be an ideal of $R$ and $n\ge0$ an integer.
Then one has $\size\syz_R^n(\mod R/I)<(n+1)(\size\syz_{R/I}^n(\mod R/I)+1)$.
In particular, if $\size\syz_{R/I}^n(\mod R/I)$ is finite, then so is $\size\syz_R^n(\mod R/I)$.
\end{cor}

\begin{proof}
This is a consequence of Proposition \ref{6}(2).
\end{proof}

\begin{lem}\label{5}
Let $M$ be an $R$-module.
Let $x\in R$ be $R$-regular.
Then $\syz_{R/xR}^n(\syz_RM/x\syz_RM)\cong\syz_R^{n+1}M/x\syz_R^{n+1}M$ for any $n\ge0$.
\end{lem}

\begin{proof}
We use induction on $n$.
Let $n>0$.
We have $\syz_{R/xR}^{n-1}(\syz M/x\syz M)\cong\syz^nM/x\syz^nM$ by the induction hypothesis, and hence $\syz_{R/xR}^n(\syz M/x\syz M)\cong\syz_{R/xR}(\syz^nM/x\syz^nM)$.
Note that $x$ is $\syz_R^nM$-regular.
There is an exact sequence $0 \to \syz^{n+1}M \to P \to \syz^nM \to 0$ of $R$-modules with $P$ projective, which gives an exact sequence $0 \to \syz^{n+1}M/x\syz^{n+1}M \to P/xP \to \syz^nM/x\syz^nM \to 0$.
Hence $\syz^{n+1}M/x\syz^{n+1}M\cong\syz_{R/xR}(\syz^nM/x\syz^nM)$.
\end{proof}

\begin{thm}\label{88}
Let $(R,\m)$ be a $d$-dimensional complete local ring with perfect coefficient field.
Then one has $\size\syz^d(\mod R)<\infty$.
Hence $\radius\syz^d(\mod R)<\infty$.
\end{thm}

\begin{proof}
We use induction on $d$.
When $d=0$, we have $\mod R=|k|_{\ell\ell(R)}$, hence $\size(\mod R)<\ell\ell(R)<\infty$.
Let $d\ge1$.
Take a filtration $R=I_0\supsetneq\cdots\supsetneq I_n=0$ of ideals such that for each $i$ one has $I_i/I_{i+1}\cong R/\p_i$ with $\p_i\in\spec R$.
Suppose that $\size\syz_{R/\p_i}^{d_i}(\mod R/\p_i)<\infty$ for all $i$, where $d_i=\dim R/\p_i$.
Then we have $\size\syz_{R/\p_i}^d(\mod R/\p_i)<\infty$, since $\syz_{R/\p_i}^d(\mod R/\p_i)$ is contained in $\syz_{R/\p_i}^{d_i}(\mod R/\p_i)$.
Corollary \ref{7} implies $\size\syz_R^d(\mod R/\p_i)<\infty$.
For each $R$-module $M$ there is an exact sequence $0 \to I_{i+1}M \to I_iM \to I_iM/I_{i+1}M \to 0$, which gives an exact sequence $0 \to \syz_R^d(I_{i+1}M) \to \syz_R^d(I_iM)\oplus R^\oplus \to \syz_R^d(I_iM/I_{i+1}M) \to 0$.
As $\syz_R^d(I_iM/I_{i+1}M)$ is in $\syz_R^d(\mod R/\p_i)$, we have $\size\syz_R^d(\mod R)<\infty$.
Thus, we may assume $R$ is a domain.

If $R$ is regular, then $\syz^d(\mod R)=|R|_1$ and $\size\syz^d(\mod R)=0<\infty$, so we may assume that $R$ is singular.
By \cite[5.15]{W} there is an ideal $J\subseteq\m$ with $\sing R=\v(J)$ and $J\Ext_R^{d+1}(\mod R,\mod R)=0$.
Since $R$ is a domain, we find an element $0\ne x\in J$.
The induction hypothesis guarantees $\syz_{R/xR}^{d-1}(\mod R/xR)\subseteq{|G|}_n^{R/xR}$ for some $R/xR$-module $G$ and an integer $n>0$.
Let $M$ be an $R$-module and put $N=\syz_R^dM$.
Note that $x$ is $N$-regular as $d>0$.
Hence $N$ is isomorphic to a direct summand of $\syz_R(N/xN)$ (cf. \cite[Lemma 2.1]{stcm}).
In view of Lemma \ref{5}, we have $N/xN\cong\syz_{R/xR}^{d-1}(\syz_RM/x\syz_RM)\in\syz_{R/xR}^{d-1}(\mod R/xR)\subseteq{|G|}_n^{R/xR}$.
Hence $N/xN$ is in $|G|_n^R$, which implies $\syz_R(N/xN)\in|\syz_RG|_n^R$.
Therefore $N$ belongs to $|\syz_RG|_n^R$, and we obtain $\syz^d(\mod R)\subseteq|\syz_RG|_n^R$.
It now follows that $\size\syz^d(\mod R)<\infty$.
\end{proof}

\begin{lem}\label{9}
Let $\A$ be an abelian category with enough projectives.
Let $0\to M\to C^0\to C^1\to\cdots\to C^{n-1}\to N\to 0$ be an exact sequence in $\A$ with $n\ge0$.
Then $M$ is in $|\syz^nN\oplus(\bigoplus_{i=0}^{n-1}\syz^iC^i)|_{n+1}$.
\end{lem}

\begin{proof}
We induce on $n$.
The case $n=0$ is trivial, so let $n\ge1$.
There are two exact sequences $0\to M\to C^0\to L\to 0$ and $0\to L\to C^1\to\cdots\to C^{n-1}\to N\to 0$.
The induction hypothesis shows $L\in|\syz^{n-1}N\oplus(\bigoplus_{i=0}^{n-2}\syz^iC^{i+1})|_n$.
A pullback diagram makes an exact sequence $0\to \syz L\to M\oplus R^\oplus \to C^0 \to 0$.
Since $\syz L$ belongs to $|\syz^nN\oplus(\bigoplus_{i=0}^{n-2}\syz^{i+1}C^{i+1})|_n$, we see that $M$ is in $|\syz^nN\oplus(\bigoplus_{i=0}^{n-1}\syz^iC^i)|_{n+1}$.
\end{proof}

\begin{cor}\label{100}
Let $R$ be a Cohen-Macaulay complete local ring with perfect coefficient field.
Then $\size\cm(R)<\infty$.
\end{cor}

\begin{proof}
As $R$ is complete, it admits a canonical module $\omega$.
Theorem \ref{88} implies $\size\syz^d(\mod R)=:n<\infty$, so we have $\syz^d(\mod R)\subseteq|G|_{n+1}$ for some $R$-module $G$.
Let $M$ be a Cohen-Macaulay $R$-module.
Then there exists an exact sequence $0\to M\to\omega^{\oplus s_0}\to\cdots\to\omega^{\oplus s_{d-1}}\to N\to0$.
It follows from Lemma \ref{9} that $M$ is in $|\syz^dN\oplus W|_{d+1}$, where $W:=\bigoplus_{i=0}^{d-1}\syz^i\omega$.
Since $\syz^dN\in|G|_{n+1}$, we have $M\in|G\oplus W|_{(n+1)(d+1)}$.
Thus $\size\cm(R)<(n+1)(d+1)<\infty$.
\end{proof}

\begin{prop}\label{12'}
Let $R$ be a Cohen-Macaulay local ring with a canonical module $\omega$.
\begin{enumerate}[\rm(1)]
\item
$\rank\cm(R)<(\dim R+1)(\size\cm(R)+1)$.
\item
$\dim\cm(R)<(\dim R+1)(\radius\cm(R)+1)$.
\end{enumerate}
In particular, one has
$$
\rank\cm(R)<\infty\Leftrightarrow\size\cm(R)<\infty\Rightarrow\dim\cm(R)<\infty\Leftrightarrow\radius\cm(R)<\infty.
$$
\end{prop}

\begin{proof}
(1) Let $n=\size\cm(R)$.
We find an $R$-module $G$ with $\cm(R)\subseteq|G|_{n+1}$.
Let $d=\dim R$ and $M\in\cm(R)$.
Similarly to the proof of Corollary \ref{100}, there exists $N\in\cm(R)$ such that $M$ is in $|\syz^dN\oplus W|_{d+1}$, where $W:=\bigoplus_{i=0}^{d-1}\syz^i\omega\in\cm(R)$.
Note that $\syz^dN\in|\syz^dG|_{n+1}$ and that $N\in\cm(R)$.
Thus we obtain $\cm(R)=|\syz^dG\oplus W|_{(n+1)(d+1)}$, and $\rank\cm(R)<(n+1)(d+1)$.

(2) In the proof of (1), replace ``$\size$'', ``$\rank$'' and ``$|\ |$'' with ``$\radius$'', ``$\dim$'' and ``$[\ ]$'', respectively.
\end{proof}

\begin{thm}\label{13}
Let $R$ be a Cohen-Macaulay local ring admitting perfect coefficient field.
Assume that either of the following holds.
\begin{enumerate}[\rm(1)]
\item
$R$ is complete.
\item
$R$ is excellent with an isolated singularity.
\end{enumerate}
Then one has $\rank\cm(R)<\infty$.
\end{thm}

\begin{proof}
(1) This assertion follows from Corollary \ref{100} and Proposition \ref{12'}(1).

(2) It follows by (1) that $\cm(\widehat R)$ has finite rank, where $\widehat R$ denotes the completion of $R$.
One can prove that $\cm(R)$ also has finite rank, making an argument similar to \cite[Remark 6.5]{LR}:

Putting $n=\rank\cm(\widehat R)$, we have $\cm(\widehat R)=|C|_{n+1}$ for some $C\in\cm(\widehat R)$.
Since $R$ is Cohen-Macaulay, excellent and with an isolated singularity, we can apply \cite[Corollary 3.6]{stcm} to see that $C$ is isomorphic to a direct summand of $\widehat G$ for some $G\in\cm(R)$.
Hence the equality $\cm(\widehat R)=|\widehat G|_{n+1}$ holds.
We establish a claim:
\begin{claim*}
Let $m>0$.
For any $N\in|\widehat G|_m$ there exists $M\in|G|_m$ such that $N$ is isomorphic to a direct summand of $\widehat M$.
\end{claim*}
\noindent
To show this claim, we use induction on $m$.
As $R$ is an isolated singularity, for all $X,Y\in\cm(R)$ the $R$-module $\Ext_R^1(X,Y)$ has finite length.
Hence there are isomorphisms $\Ext_R^1(X,Y)\cong\Ext_R^1(X,Y)^{\widehat{\ \ }}\cong\Ext_{\widehat R}^1(\widehat X,\widehat Y)$, which imply that every short exact sequence $0\to \widehat Y\to E \to\widehat X\to0$ of $\widehat R$-modules is isomorphic to the completion of some short exact sequence $0\to Y\to E' \to X\to0$ of $R$-modules.
The claim follows from this.
Using this claim and \cite[Lemma 5.7]{ddc}, we observe that $\cm(R)=|G|_{n+1}$ holds.
Therefore $\rank\cm(R)\le n<\infty$.
\end{proof}


\section{Some discussions and open questions}\label{appsec}

In this section we  relate our results to 
the uniform Auslander condition and discuss some open questions.
For a local ring $R$, Jorgensen and \c{S}ega \cite{JS} introduced the {\it uniform Auslander condition}:
$$
\begin{array}{rl}
\uac:\ & \text{There exists an integer $n$ such that for all $R$-modules $M,N$ with} \\
& \text{$\Ext^i_R(M,N)=0$ for all $i\gg0$ one has $\Ext_R^i(M,N)=0$ for all $i\ge n$}.
\end{array}
$$
It is known that this condition is satisfied if the local ring $R$ is a complete intersection, a Golod ring, a Gorenstein ring with $\mult R=\codim R+2$, or a Gorenstein ring with $\codim R\le4$.
Here $\mult R$ denotes the multiplicity of $R$.
These are proved in \cite[Proposition 1.4]{JS}, \cite[Theorem 4.7]{AB'}, \cite[Theorem 3.5]{HJ} and \cite[Theorem 3.4]{S'}, respectively.
More information can be found in \cite[Appendix A]{CH}.
On the other hand, there exists an example of a Gorenstein local ring which does not satisfy \uac; see \cite[Theorem in \S0]{JS}.

The result below says that over a Gorenstein local ring the condition \uac\ is closely related to the thickness of resolving subcategories of Cohen-Macaulay modules.

\begin{prop}\label{55}
Let $R$ be a Gorenstein local ring.
Assume every resolving subcategory of $\mod R$ contained in $\CM(R)$ is a thick subcategory of $\CM(R)$.
Then $R$ satisfies \uac.
\end{prop}

\begin{proof}
Let $t\ge0$ be an integer, and let $M,N$ be $R$-modules with $\Ext_R^i(M,N)=0$ for all $i>t$.
We define a subcategory $\X$ of $\mod R$ to consist of all Cohen-Macaulay $R$-modules $X$ satisfying $\Ext_R^i(X,N)=0$ for all $i>t$.
Then $\X$ is a resolving subcategory of $\mod R$ contained in $\CM(R)$.
By assumption, $\X$ is a thick subcategory of $\CM(R)$.
Set $d=\dim R$.
Since $\syz^dM$ is in $\X$, so is $\syz^{-t}\syz^dM$.
We have
$
\Ext_R^i(M,N) \cong\Ext_R^{i-d}(\syz^dM,N)\cong\Ext_R^{i-d}(\syz^t(\syz^{-t}\syz^dM),N)
 \cong\Ext_R^{i-d+t}(\syz^{-t}\syz^dM,N)=0
$
for all integers $i>d$.
\end{proof}

There is also a connection between thickness of resolving subcategories of totally reflexive modules and closure under $R$-duals.
Here we say that a subcategory $\X$ of $\mod R$ is {\em closed under $R$-duals} if for each module $M$ in $\X$ its $R$-dual $M^\ast$ is also in $\X$.

\begin{prop}\label{66}
\begin{enumerate}[\rm(1)]
\item
Let $R$ be local.
Let $\X$ be a resolving subcategory of $\mod R$ contained in $\TR(R)$.
If $\X$ is closed under $R$-duals, then $\X$ is a thick subcategory of $\TR(R)$.
\item
Let $R$ be a local hypersurface.
Then every resolving subcategory of $\mod R$ contained in $\CM(R)$ is closed under $R$-duals.
\end{enumerate}
\end{prop}

\begin{proof}
(1) According to Proposition \ref{12}, we have only to show that $\X$ is closed under cosyzygies.
Let $X\in\X$.
There is an exact sequence $0 \to \syz(X^*) \to F \to X^* \to 0$, where $F$ is free.
Dualizing this by $R$, we get an exact sequence
$
0 \to X \to F^* \to (\syz(X^*))^* \to 0.
$
Note that $(\syz(X^*))^*$ is totally reflexive.
We easily see that $(\syz(X^*))^*$ is isomorphic to $\syz^{-1}X$ up to free summand.
As $\X$ is a resolving subcategory closed under $R$-duals, $(\syz(X^*))^*$ belongs to $\X$, and so does $\syz^{-1}X$.

(2) It follows from \cite[Main Theorem]{stcm} that every resolving subcategory of $\mod R$ contained in $\CM(R)$ can be described as $\NF^{-1}_{\mathrm{CM}}(W)$, where $W$ is a specialization-closed subset of $\Spec R$ contained in $\Sing R$.
If $M$ is an $R$-module in $\NF^{-1}_{\mathrm{CM}}(W)$, then we have $\NF(M^*)=\NF(M)\subseteq W$, which shows that $\NF^{-1}_{\mathrm{CM}}(W)$ also contains $M^*$.
\end{proof}

Now we have reached the following question.

\begin{ques}\label{780845}
Let $R$ be a Gorenstein local ring.
Let us consider the following five conditions.
\begin{enumerate}[(1)]
\item
$R$ is a complete intersection.
\item
Every resolving subcategory of $\mod R$ contained in $\CM(R)$ is closed under $R$-duals.
\item
Every resolving subcategory of $\mod R$ contained in $\CM(R)$ is a thick subcategory of $\CM(R)$.
\item
$R$ satisfies \uac.
\item
Conjecture \ref{2} is true for $R$.
\end{enumerate}
We know that the implications $(2)\Rightarrow (3)$ and $(1) \Rightarrow (3) \Rightarrow (4)$ hold by Propositions \ref{55}, \ref{66}(1) and Corollary \ref{1559}.
The implication $(1)\Rightarrow(2)$ is also true if $R$ is a hypersurface by Proposition \ref{66}(2).
Very recently, motivated by the first version of the present paper, Stevenson \cite{Ste2} proved that the implication $(1)\Rightarrow(2)$ holds in the case where $R$ is a quotient of a regular local ring.
Corollary \ref{8} says that $(3) \Rightarrow (5)$ holds if $R$ is Gorenstein.
How about the other implications among these five conditions?
\end{ques}

\begin{rem}\label{r12}
According to a recent preprint by Stevenson \cite{S} (see also \cite{I}), if $R$ is a quotient of a regular local ring by a regular sequence, then one can classify the thick subcategories of $\uCM(R)$ in terms of ``support varieties".
Thus, one can also classify the resolving subcategories of $\mod R$ contained in $\CM(R)$ by using Corollary \ref{1559} and \cite[Proposition 6.2]{stcm}.
In relation to this, the resolving subcategories over a regular ring can be classified completely.
This classification theorem is stated and proved in \cite{LR2}.
\end{rem}

\section*{Acknowledgments}
The authors would like to thank Luchezar Avramov, Jesse Burke, Craig Huneke, Osamu Iyama, Srikanth Iyengar and Greg Stevenson for their valuable comments.
This work was done during the visits of the second author to University of Kansas in May, July and August, 2011, and July, 2012.
He is grateful for their kind hospitality.
The authors also thank the referee for his/her careful reading and helpful comments.


\begin{thebibliography}{99}
\bibitem{ddc}
T. Aihara; R. Takahashi, \emph{Generators and dimensions of derived categories}, Preprint (2011), \texttt{arXiv:1106.0205}.
\bibitem{AIR}
T. Araya; K. Iima; R. Takahashi, \emph{On the structure of Cohen-Macaulay modules over hypersurfaces of countable Cohen-Macaulay representation type}, J. Algebra 361 (2012), 213--224.
\bibitem{A}
M. Auslander, Anneaux de Gorenstein, et torsion en alg\`{e}bre commutative, S\'{e}minaire d'Alg\`{e}bre Commutative dirig\'{e} par Pierre Samuel, 1966/67, Texte r\'{e}dig\'{e}, d'apr\`{e}s des expos\'{e}s de Maurice Auslander, Marquerite Mangeney, Christian Peskine et Lucien Szpiro, \'{E}cole Normale Sup\'{e}rieure de Jeunes Filles, Secr\'{e}tariat math\'{e}matique, Paris, 1967.
\bibitem{AB}
M. Auslander; M. Bridger, \emph{Stable module theory}, Mem. Amer. Math. Soc. 94 (1969).
\bibitem{AB''}
M. Auslander; R.-O. Buchweitz, \emph{The homological theory of maximal Cohen-Macaulay approximations}, Colloque en l'honneur de Pierre Samuel (Orsay, 1987), M\'{e}m. Soc. Math. France (N.S.) 38 (1989), 5--37.
\bibitem{AB'}
L. L. Avramov; R.-O. Buchweitz, \emph{Support varieties and cohomology over complete intersections}, Invent. Math. 142 (2000), no. 2, 285--318.
\bibitem{ABIM}
L. L. Avramov; R.-O. Buchweitz; S. B. Iyengar; C. Miller, \emph{Homology of perfect complexes}, Adv. Math. 223 (2010), no. 5, 1731--1781.
\bibitem{AGP}
L. L. Avramov; V. N. Gasharov; I. V. Peeva, \emph{Complete intersection dimension}, Inst. Hautes \'{E}tudes Sci. Publ. Math. 86 (1997), 67--114 (1998).
\bibitem{Be}
P. A. Bergh, \emph{Modules with reducible complexity}, J. Algebra 310 (2007), no. 1, 132--147.
\bibitem{B}
R.-O. Buchweitz, Maximal Cohen-Macaulay modules and Tate-cohomology over Gorenstein rings, preprint (1986), \texttt{http://hdl.handle.net/1807/16682}.

\bibitem{Ch}
L. W. Christensen, \emph{Gorenstein dimensions}, Lecture Notes in Mathematics, 1747, Springer-Verlag, Berlin, 2000.
\bibitem{CH}
L. W. Christensen; H. Holm, \emph{Algebras that satisfy Auslander's condition on vanishing of cohomology}, Math. Z. 265 (2010), no. 1, 21--40.
\bibitem{LR}
H. Dao; R. Takahashi, \emph{The dimension of a subcategory of modules}, preprint (2012), \texttt{arXiv:1203.1955}.
\bibitem{LR2}
H. Dao; R. Takahashi, \emph{Classification of resolving subcategories and grade consistent functions}, Int. Math. Res. Not. IMRN (to appear).
\bibitem{H}
D. Happel, \emph{Triangulated categories in the representation theory of finite-dimensional algebras}, London Mathematical Society Lecture Note Series, 119, Cambridge University Press, Cambridge, 1988.

\bibitem{HJ}
C. Huneke; D. A. Jorgensen, \emph{Symmetry in the vanishing of Ext over Gorenstein rings}, Math. Scand. 93 (2003), no. 2, 161--184.
\bibitem{I}
S. B. Iyengar, Stratifying derived categories associated to finite groups and commutative rings, Kyoto RIMS Workshop on Algebraic Triangulated Categories and Related Topics, \texttt{http://www.math.unl.edu/\~{}siyengar2/Papers/RIMS0709.pdf}.
\bibitem{JS}
D. A. Jorgensen; L. \c{S}ega, \emph{Nonvanishing cohomology and classes of Gorenstein rings}, Adv. Math. 188 (2004), no. 2, 470--490.
\bibitem{M}
H. Minamoto, \emph{A note on dimension of triangulated categories}, Proc. Amer. Math. Soc. (to appear).
\bibitem{R}
R. Rouquier, \emph{Dimensions of triangulated categories}, J. K-Theory 1 (2008), 193--256.
\bibitem{S'}
L. M. \c{S}ega, \emph{Vanishing of cohomology over Gorenstein rings of small codimension}, Proc. Amer. Math. Soc. 131 (2003), no. 8, 2313--2323.
\bibitem{S}
G. Stevenson, \emph{Subcategories of singularity categories via tensor actions}, Compos. Math. (to appear).
\bibitem{Ste2}
G. Stevenson, \emph{Duality for bounded derived categories of complete intersections}, preprint (2012), \texttt{arXiv:1206.2724}.
\bibitem{res}
R. Takahashi, \emph{Modules in resolving subcategories which are free on the punctured spectrum}, Pacific J. Math. 241 (2009), no. 2, 347--367.
\bibitem{stcm}
R. Takahashi, \emph{Classifying thick subcategories of the stable category of Cohen-Macaulay modules}, Adv. Math. 225 (2010), no. 4, 2076--2116.
\bibitem{crs}
R. Takahashi, \emph{Classifying resolving subcategories over a Cohen-Macaulay local ring}, Math. Z. 273 (2013), no. 1, 569--587.
\bibitem{W}
H.-J. Wang, \emph{On the Fitting ideals in free resolutions}, Michigan Math. J. 41 (1994), no. 3, 587--608.

\bibitem{Y}
M. Yoshiwaki, \emph{On selfinjective algebras of stable dimension zero}, Nagoya Math. J. 203 (2011), 101--108.

\bibitem{Y2}
Y. Yoshino, \emph{A functorial approach to modules of G-dimension zero}, Illinois J. Math. 49 (2005), no. 2, 345--367.
\end{thebibliography}
\end{document}